\documentclass[journal]{IEEEtran}
\usepackage{amssymb,amsmath,amsfonts}
\usepackage{algorithm,algorithmic}
\usepackage{blkarray}
\usepackage{bbm}
\usepackage{bm}
\usepackage[makeroom]{cancel}
\usepackage{cases}
\usepackage{cite}
\usepackage{changebar}
\usepackage[usenames,dvipsnames]{color}
\usepackage{dsfont}
\usepackage{epsfig}
\usepackage[T1]{fontenc}
\usepackage{float}
\usepackage{flushend}
\usepackage{graphicx}
\usepackage{hyperref} 
\usepackage[latin9]{inputenc}
\usepackage{mathtools}
\usepackage{mathrsfs}
\usepackage{setspace}
\usepackage{soul} 
\usepackage{subfigure}
\usepackage{stfloats}
\usepackage{tikz}
\usepackage{url}
\usepackage{verbatim}
\usepackage{xspace}

\usetikzlibrary{automata,chains,arrows}
\usetikzlibrary{arrows.meta}

\floatstyle{ruled}
\newfloat{model}{H}{mod}
\floatname{model}{\footnotesize Model}
\newfloat{notatio}{H}{not}
\floatname{notatio}{\footnotesize Notation}

\newtheorem{remark}{\bfseries Remark}
\newtheorem{theorem}{\bfseries Theorem}

\newtheorem{assumption}{\bfseries Assumption}

\newtheorem{proposition}{\bfseries Proposition}

\newtheorem{definition}{\bfseries Definition}

\newcommand{\T}{^{\mbox{\tiny T}}}
\def\R{\mathbb{R}}
\def\Z{\mathbb{Z}}

\let\mathbb=\mathds 
\def \defn{\stackrel{\triangle}{=}}

\def \defn{\stackrel{\triangle}{=}}
\floatstyle{ruled}
\newfloat{model}{H}{mod}
\floatname{model}{\footnotesize Model}
\newfloat{notatio}{H}{not}
\floatname{notatio}{\footnotesize Notation}

\newenvironment{list4}{
  \begin{list}{$\bullet$}{
      \setlength{\itemsep}{0.05cm}
      \setlength{\labelsep}{0.2cm}
      \setlength{\labelwidth}{0.3cm}
      \setlength{\parsep}{0in}
      \setlength{\parskip}{0in}
      \setlength{\topsep}{0in}
      \setlength{\partopsep}{0in}
      \setlength{\leftmargin}{0.21in}}}
      {\end{list}}
\newenvironment{list5}{
  \begin{list}{$\bullet$}{
      \setlength{\itemsep}{0.05cm}
      \setlength{\labelsep}{0.2cm}
      \setlength{\labelwidth}{0.3cm}
      \setlength{\parsep}{0in}
      \setlength{\parskip}{0in}
      \setlength{\topsep}{0in}
      \setlength{\partopsep}{0in}
      \setlength{\leftmargin}{0.17in}}}
      {\end{list}}

\definecolor{themisblue}{rgb}{0.76, 0.89, 1.0}

\begin{document}
%
%
%
%
\title{\LARGE \bf Fully Distributed Alternating Direction Method of Multipliers \\ in Digraphs {via Finite-Time Termination Mechanisms}}

\author{Wei Jiang and Themistoklis Charalambous,~\IEEEmembership{Senior Member, IEEE}
\thanks{W. Jiang and T. Charalambous are with the Department of Electrical Engineering and Automation, School of Electrical Engineering, Aalto University, Espoo, Finland.
Emails: {\tt \{name.surname\}@aalto.fi}.}
\thanks{Preliminary results of this work have been {accepted} in the European Control Conference (ECC) \cite{ECC2021:Wei-Themis}. In the journal version, we propose a new algorithm in which no global information is needed about the size of the network in order to be able to run the distributed ADMM algorithm (unlike the algorithm proposed in our paper in ECC). In addition, we show that  if the global objective function is strongly convex and smooth, the proposed algorithms have an ``approximate'' R-linear  convergence rate.}
}


\maketitle


%
%
%
%
\begin{abstract}
In this work, we consider the distributed optimization problem in which each node has its own convex cost function and can communicate directly only with its neighbors, as determined by a directed communication topology (directed graph or digraph). First, we reformulate the optimization problem so that  Alternating Direction Method of Multipliers (ADMM) can be utilized. Then, we propose an algorithm, herein called Distributed Alternating Direction Method of Multipliers using Finite-Time Exact Ratio Consensus (D-ADMM-FTERC), to solve the multi-node convex optimization problem, in which every node performs iterative computations and exchanges information with its neighbors. At every iteration of D-ADMM-FTERC, each node solves a local convex optimization problem for the one of the primal variables and utilizes a finite-time \emph{exact} consensus protocol to obtain the optimal value of the other variable, since the cost function for the second primal variable is not decomposable. 
Since D-ADMM-FTERC requires to know the upper bound on the number of nodes in the network, we furthermore propose a { new algorithm,} called Fully D-ADMM Finite-Time Distributed Termination (FD-ADMM-FTDT) {algorithm,} which do{es} not need any global information.
If the individual cost functions are convex and not-necessarily differentiable, the proposed algorithms converge at a rate of $O(1/k)$, where $k$ is the iteration counter. 
Additionally, if the global objective function is strongly convex and smooth, the proposed algorithms {have} an ``approximate'' R-linear  convergence rate. 
The efficacy of FD-ADMM-FTDT is demonstrated via a distributed $ l_1 $ regularized logistic regression optimization example. Additionally, comparisons with other state-of-the-art algorithms are provided on large-scale networks showing the superior precision and time-efficient performance of FD-ADMM-FTDT.
\end{abstract}


\IEEEpeerreviewmaketitle
%
%
%
%
\section{Introduction}
\label{sec:introduction}

\subsection{Motivation}

\IEEEPARstart{T}{he} main objective is the solution of an additive cost optimization problem over a digraph in a distributed fashion, where each individual cost is known solely to the node; this type of problems is often referred to as distributed optimization problem and a wide variety of engineering problems (e.g., wireless sensor networks \cite{2004:Sensors} and machine learning \cite{boyd2011distributed,Nedic2020Magazine}) fall within this framework. For this reason, even though such problems were targeted already in the 80's \cite{1986:tsitsiklis,1989:BnT}, the field of distributed optimization has attracted a lot of attention by the research community again recently; see, for example,~\cite{phdthesis:Johansson2008,2009:Johansson, nedic2009distributed, makhdoumi2017convergence, Rabbat:ProceedingsIEEE2018, Tao:2019} and references therein.

\vspace{-0.3cm}
\subsection{Related Work}

There are two main research strands for solving distributed optimization methods in the literature: \emph{(i)} primal and \emph{(ii)} dual-based optimization methods. Our work falls in the strand of dual-based optimization methods, and more specifically on distributed approaches for realizing the ADMM. In that direction, there are two main communication topologies considered: \emph{(i)} master-workers communication topology and \emph{(ii)} multi-node communication topology. When the ADMM has a master-worker communication topology, the worker nodes optimize their local objectives and communicate their local variables to the master node which updates the global optimization variable and send it back to the workers. When the ADMM has no master node, the optimization problem is solved over a network of nodes. Here in, we focus on the ADMM realized on multi-node communication topologies.

There have been several ADMM algorithms proposed for the case which the multi-node communication topology assumes that every communication link is bidirectional, thus forming a communication topology represented by an undirected graph; see, for example, \cite{wei2012distributed,ADMM_Shi:2014,makhdoumi2017convergence,falsone2020tracking}. In the case for which some communication links are not necessarily bidirectional, these approaches fail to converge to the optimal solution. Distributed ADMM approaches for digraphs are very limited. The first distributed ADMM approach for directed graphs with convergence guarantees \cite{khatana2020d}, and the inspiration for this work, proposes a consensus-based approach to compute one of the primal variables of ADMM for digraphs. Specifically, at every step, while one of the primal variables and the Lagrange multiplier are computed at the node itself, the other primal variable is \emph{approximated} by running a consensus algorithm that produces asymptotic convergence for a finite number of steps and, as a consequence, an approximate solution at every optimization step is obtained. 

\subsection{Our Contributions}
The contributions of our paper are the following:
\begin{list5}
\item[1)] First, we enhance a distributed protocol proposed in \cite{themisCDC:2013,themisTCNS:2015} and used in D-ADMM-FTERC, with a distributed termination mechanism \cite{charalambous2018stop}, with which each node in a digraph can compute the average consensus over a \emph{minimal} number of steps and agree with the other nodes in the network when to terminate their iterations, provided they have all computed their exact average. More specifically, we modify the distributed termination mechanism \cite{charalambous2018stop} to allow for the nodes to synchronize the optimization steps, without requiring any global information.
\item[2)] Next, we propose a distributed ADMM algorithm, FD-ADMM-FTDT, that solves \emph{exactly} the multi-node convex optimization problem in digraphs. At every iteration of the algorithm, each node solves a local convex optimization problem for one of the primal variables by utilizing the proposed fully distributed finite-time consensus protocol to compute the \emph{exact} optimal of the other primal variable. 
\item[3)] FD-ADMM-FTDT performance is evaluated via extensive simulations and compared with the only other known distributed ADMM approach suitable for directed graphs \cite{khatana2020d}, and it is shown that our approach apart from computing the exact solution (unlike the one in \cite{khatana2020d}), it requires fewer iterations per optimization step and therefore the optimization speed is accelerated.
\end{list5}

\subsection{Organization}

The remainder of the paper is organized as follows. In Section \ref{sec:preliminaries}, we provide necessary notation and background knowledge for the development of our results. In Section \ref{sec:formulation}, the problem to be solved is formulated, and in Section \ref{sec:results} our proposed algorithms are explained. Illustrative examples are presented in Section~\ref{sec:examples}. Finally, Section \ref{sec:conclusions} presents concluding remarks and future directions.

%
%
%
%
\section{Notation and Preliminaries}\label{sec:preliminaries}

\subsection{Notation}

The set of real (integer) numbers is denoted by $\mathds{R}$ ($\mathds{Z}$) and the set of positive numbers (integers) is denoted by $\mathds{R}_{+}$ ($\mathds{Z}_{+}$).  $\mathds{R}^n_+$ denotes the non-negative orthant of the $n$-dimensional real space $\mathds{R}^n$. Vectors are denoted by small letters whereas matrices are denoted by capital letters.  $A\T$ denotes the transpose of matrix $A$. The $i^{th}$ component of a vector $x$ is denoted by $x_i$, and the notation $x\geq y$ implies that $x_i\geq y_i$ for all components $i$. For $A\in \mathbb{R}^{n\times n}$, $a_{ij}$ denotes the entry in row $i$ and column $j$.
By $\mathbb{1}$ we denote the all-ones vector and by $I$ we denote the identity matrix (of appropriate dimensions). We also denote by $e\T _j=[0,\ldots,0, 1_{j^{th}}, 0, \ldots, \ldots, 0] \in {\R}^{1 \times n}$, where the single ``1'' entry is at the $j^{\textrm{th}}$ position. $|A|$ is the element-wise absolute value of matrix A (i.e., $|A|\triangleq[|A_{ij}|]$), $A \leq B$ ($A<B$) is the (strict) element-wise inequality between matrices $A$ and $B$.
A matrix whose elements are nonnegative, called nonnegative matrix, is denoted by $A \geq 0$ and a matrix whose elements are positive, called positive matrix, is denoted by $A>0$.
$ \|x\| $ denotes the Euclidean norm of $ x $.
$ \langle a, b\rangle $ denotes the usual Euclidean inner product $ a\T b $.
{Suppose  that the sequence $ \{x_k\} $ converges to $ x^{*} $. The sequence is said to converge Q-linearly to $ x^{*} $ if there exists a number $  \mu \in (0,1) $  such that $ \|x^{k+1}- x^{*} \|/ \|x^{k}- x^{*} \| \le \mu $. The sequence is said to converge R-linearly to $ x^{*} $ if there exists a sequence $ \{\varepsilon _{k}\} $ such that  $ \|x^{k}- x^{*} \| \le \varepsilon _{k} $ and  $ \{\varepsilon _{k}\} $  converges Q-linearly to zero.}

In multi-component systems with fixed communication links (edges), the exchange of information between components (nodes) can be conveniently captured by a directed graph (digraph) $\mathcal{G}(\mathcal{V}, \mathcal{E})$ of order $n$ $(n \geq 2)$, where $\mathcal{V} = \{v_1,v_2,\ldots,v_n\}$ is the set of nodes and $\mathcal{E} \subseteq \mathcal{V} \times \mathcal{V}$ is the set of edges. A directed edge from node $v_i$ to node $v_j$ is denoted by $\varepsilon_{ji} = (v_j, v_i)\in \mathcal{E}$ and represents a communication link that allows node $v_j$ to receive information from node $v_i$. A graph is said to be undirected if and only if $\varepsilon_{ji} \in \mathcal{E}$ implies $\varepsilon_{ij}  \in \mathcal{E}$. A digraph is called \emph{strongly} connected if there exists a path from each vertex $v_i$ of the graph to each vertex $v_j$ ($v_j \neq v_i$). In other words, for any $v_j, v_i \in \mathcal{V}$, $v_j\neq v_i$, one can find a sequence of nodes $v_i = v_{l_1}$, $v_{l_2}$, $v_{l_3}$, $\ldots$, $v_{l_m}=v_j$ such that link $(v_{l_{s+1}}, v_{l_{s}}) \in \mathcal{E}$ for all $s = 1, 2, \ldots, m-1$. The diameter $D$ of a graph is the longest shortest path between any two nodes in the network.

All nodes that can transmit information to node $v_j$ directly are said to be in-neighbors of node $v_j$ and belong to the set $\mathcal{N}^{-}_j=\{ v_i \in \mathcal{V} \; | \; \varepsilon_{ji} \in \mathcal{E} \}$. The cardinality of $\mathcal{N}^{-}_j$, is called the \emph{in-degree} of $v_j$ and is denoted by $\mathcal{D}^{-}_{j}=\left| \mathcal{N}^{-}_j \right|$. The nodes that receive information from node $v_j$ belong to the set of out-neighbors of node $v_j$, denoted by $\mathcal{N}^{+}_j=\{ v_l \in \mathcal{V} \; | \; \varepsilon_{lj} \in \mathcal{E} \}$. The cardinality of $\mathcal{N}^{+}_j$, is called the \emph{out-degree} of $v_j$ and is denoted by $\mathcal{D}^{+}_{j}= \left| \mathcal{N}^{+}_j \right|$.

\vspace{-0.3cm}
\subsection{Average Consensus}


In the type of algorithms we consider, we associate a positive weight $p_{ji}$ for each edge $\varepsilon_{ji} \in \mathcal{E} \cup \{ (v_j, v_j) \; | \: v_j \in \mathcal{V} \}$. The nonnegative matrix $P = [p_{ji} ] \in \mathbb{R}_{+}^{n\times n}$ (with $p_{ji}$ as the entry at its $j$th row, $i$th column position) is a weighted adjacency matrix (also referred to as weight matrix) that has zero entries at locations that do not correspond to directed edges (or self-edges) in the graph. In other words, apart from the main diagonal, the zero-nonzero structure of the adjacency matrix $P$ matches exactly the given set of links in the graph. In a synchronous setting, each node $v_j$ updates and sends its information to its neighbors at discrete times $T_0, T_1, T_2, \ldots$. We index nodes' information states and any other information at time $T_t$ by $t$. Hence, we use $w_j^t \equiv w_j[t]  \in \mathbb{R}$ to denote the information state of node $j$ at time $T_t$. Note that ${w_j^{t}}\T$ denotes (is equivalent to) $w_j[t]\T$.

Each node updates its information state $w_{j}^t$ by combining the available information received by its neighbors $w_{i}^t$ ($v_i\in \mathcal{N}^{-}_j$) using the positive weights $p_{ji}^t$, that capture the weight of the information inflow from node $v_i$ to node $v_j$ at time $t$. In this work, we assume that each node $v_j$ can choose its self-weight and the weights on its out-going links $\mathcal{N}^{+}_j $ only. Hence, in its general form, each node updates its information state according to the following relation:
\begin{align}\label{eq:1_1}
w_{j}^{t+1} =p_{jj}w_{j}^t + \sum_{v_i \in \mathcal{N}^{-}_j}  p_{ji} w_{i}^t \;, \ \ k\geq 0\;,
\end{align}
\noindent where  $w_{j}^0 \in \mathbb{R}$ is the initial state of node $v_j$. If we let $w^t=(w_1^t \ \ w_2^t \ \ \ldots  \ \  w_n^t )\T$ and $P= [p_{ji} ] \in \mathbb{R}_{+}^{n\times n}$, then \eqref{eq:1_1} can be written in matrix form as
\begin{align}\label{eq:2_1}
w^{t+1} =Pw^t ,
\end{align}
\noindent where $w^0 = (w_1^0 \ \ w_2^0 \ \ \ldots  \ \  w_n^0 )\T \triangleq w_0$. We say that the nodes asymptotically reach average consensus if
$$
\lim_{t \rightarrow \infty} w_j^t = \frac{\sum_{v_i\in\mathcal{V}} w_i^0}{n}\; , \quad \forall v_j\in\mathcal{V} \; .
$$
The necessary and sufficient conditions for \eqref{eq:2_1} to reach average consensus are the following: (a) $P$ has a simple eigenvalue at one with left eigenvector  $\mathbb{1}\T$ and right eigenvector $\mathbb{1}$, and (b) all other eigenvalues of $P$ have magnitude less than $1$. If $P\geq 0$ (as in our case), the necessary and sufficient condition is that $P$ is a primitive doubly stochastic matrix. In an undirected graph, assuming each node knows $n$ (or an upper bound $n{'}$) and the graph is connected, each node $v_j$ can distributively choose the weights on its outgoing links to be $\frac{1}{n'}$ and set its diagonal to be $1-\frac{\mathcal{D}_j^+}{n'}$ (where $\mathcal{D}_j^+ = \mathcal{D}_j^- \triangleq \mathcal{D}_j$), so that the resulting $P$ is primitive doubly stochastic. However, this weight selection does not necessarily yield a doubly stochastic weight matrix in a digraph.

\vspace{-0.3cm}
\subsection{Ratio Consensus}\label{preliminary_rc}

In \cite{2010:christoforos}, an algorithm is suggested, called ratio consensus, that solves the average consensus problem in a directed graph in which each node $v_j$ distributively sets the weights on its self-link and outgoing-links to be $\frac{1}{1+\mathcal{D}_j^+}$, so that the resulting weight matrix $P$ is column stochastic, but not necessarily row stochastic. Average consensus is reached by using this weight matrix to run two iterations with appropriately chosen initial conditions. The algorithm is stated below for a specific choice of weights on each link that assumes that each node knows its out-degree; note, however, that the algorithm  works for any set of weights that adhere to the graph structure and form a primitive column stochastic weight matrix.

\begin{proposition}[\hspace{-0.001cm \cite{2010:christoforos}}]\label{lemma_christoforos}
Consider a strongly connected digraph $\mathcal{G}(\mathcal{V}, \mathcal{E})$. Let $y_j^t$ and $x_j^t$ (for all $v_j \in \mathcal{V}$ and $t=0,1,2,\ldots$) be the result of the iterations
\begin{subequations}\label{eq:1}
\begin{align}
y_j^{t+1}=p_{jj} y_j^t+ \sum_{v_i \in \mathcal{N}^{-}_j} p_{ji} y_i^t \; , \label{subeq:1} \\
x_j^{t+1}=p_{jj} x_j^t+ \sum_{v_i \in \mathcal{N}^{-}_j} p_{ji} x_i^t \; , \label{subeq:2}
\end{align}
\end{subequations}
where $p_{lj} = \frac{1}{1 + \mathcal{D}_j^+}$ for $v_l \in \mathcal{N}_j^+ \cup \{ v_j \}$ (zeros otherwise), and the initial conditions are $y^0=y_0$ and  $x^0=\mathbb{1}$. Then, the solution to the average consensus problem can be asymptotically obtained as
$
\displaystyle \lim_{t\rightarrow \infty} \mu_j^t=\frac{\sum_{v_i\in \mathcal{V}} y_i^0}{|\mathcal{V}|} \; , \forall v_j \in \mathcal{V} \; ,
$
where
$
\displaystyle \mu_j^t=\frac{y_j^t}{x_j^t} \; .
$
\end{proposition}
\begin{remark}
Proposition~\ref{lemma_christoforos} proposes a decentralised algorithm with which the exact average is \emph{asymptotically} reached, even if the directed graph is not balanced. 
\end{remark}

\subsection{Finite-Time Exact Ratio Consensus (FTERC)}\label{FTERC}

In what follows, we present a distributed protocol proposed in \cite{themisCDC:2013,themisTCNS:2015} with which each node can compute, based on its own local observations and after a minimal number of steps, the exact average. This protocol is based on the algorithm in Proposition~\ref{lemma_christoforos}, with which every node can compute $\mu_j \triangleq \lim_{t\rightarrow \infty} \mu_j^t$ in a \emph{minimum} number of steps.

\begin{definition}(Minimal polynomial of a matrix pair)
The minimal polynomial associated with the matrix pair $[P,e\T_j]$ denoted by
$q_j(s)=s^{M_j+1}+\sum_{i=0}^{M_j} \alpha^{(j)}_i s^i$ is the monic polynomial of minimum degree $M_j+1$ that
satisfies $e\T_j q_j(P)=0$.
\end{definition}
Considering the iteration in~\eqref{eq:2_1} with weight matrix $P$, it is easy to show (e.g., using the techniques in \cite{2009:Ye}) that
\begin{equation}\label{regression}
\sum _{i=0}^{M_j+1} \alpha^{(j)}_i w_j^{t+i}=0, \quad \forall t \in \mathds{Z}_{+} \; ,
\end{equation}
where $\alpha^{(j)}_{M_j+1}=1$. Let us now denote the $z$-transform of $w_j^t$ as $W_j(z) \defn \Z(w_j^t)$. From \eqref{regression} and the time-shift property of the $z-$transform, it is easy to show (see \cite{2009:Ye,2013:Ye})
\begin{equation} \label{ztranform}
W_j(z)=\frac{\sum _{i=1}^{M_j+1} \alpha^{(j)}_i \sum _{\ell=0}^{i-1}w_j^{\ell} z^{i-\ell}}{q_j(z)} \; ,
\end{equation}
where $q_j(z)$ is the minimal polynomial of $[P,e\T_j]$. If the network is strongly connected, $q_j(z)$ does not have any unstable poles apart from one at $1$; we can then define the following polynomial:
\begin{align}\label{beta}
p_j(z) \triangleq \frac{q_j(z)}{z-1} \triangleq \sum_{i=0}^{M_j} \beta^{(j)}_i z^i\; .
\end{align}

The application of the final value theorem \cite{2009:Ye,2013:Ye} yields:
\begin{subequations}\label{eq:phi}
\begin{align}
\phi_{y}(j)=\lim_{t\rightarrow \infty}y_j^t=\lim_{z\rightarrow 1}(z-1)Y_j(z)=\frac{y_{M_j}\T {\bm \beta}_j}{\mathbb{1}\T {\bm \beta}_j} \; , \label{phi:1} \\
\phi_{x}(j)=\lim_{t\rightarrow \infty}x_j^t=\lim_{z\rightarrow 1}(z-1)X_j(z)=\frac{x_{M_j}\T {\bm \beta}_j}{\mathbb{1}\T {\bm \beta}_j} \; , \label{phi:2}
\end{align}
\end{subequations}
where $y\T_{M_j}=(y_j^0,y_j^1,\ldots,y_j^{M_j})$, $x\T_{M_j}=(x_j^0,x_j^1,\ldots,x_j^{M_j})$
and ${\bm \beta}_j$ is the vector of coefficients of the polynomial $p_j(z)$.

Consider the vectors of $2t+1$ successive discrete-time values at node $v_j$, given by
\begin{align*}
y\T_{2t} &= (y_j^0, y_j^1, \ldots, y_j^{2t} ), \\
x\T_{2t} &= (x_j^0, x_j^1, \ldots, x_j^{2t} ),
\end{align*}
for the two iterations $y_j^t$ and $x_j^t$ at node $v_j$ (as given in iterations \eqref{subeq:1} and \eqref{subeq:2}), respectively.
Let us define their associated Hankel matrices:
\[
\Gamma\{y\T_{2t}\}\triangleq \begin{bmatrix}
y_j^0 & y_j^1 & \ldots & y_j^t \\
y_j^1 & y_j^2 & \ldots & y_j^{t+1} \\
\vdots & \vdots & \ddots & \vdots \\
y_j^t & y_j^{t+1} & \ldots & y_{j}^{2t}
\end{bmatrix},
\]
\[
\Gamma\{x\T_{2t}\}\triangleq \begin{bmatrix}
x_j^0 & x_j^1 & \ldots & x_j^t \\
x_j^1 & x_j^2 & \ldots & x_j^{t+1} \\
\vdots & \vdots & \ddots & \vdots \\
x_j^t & x_j^{t+1} & \ldots & x_{j}^{2t}
\end{bmatrix}.
\]
We also consider the vector of differences between successive values of $y_j^t$
and $x_j^t$:
\[
\overline{y}\T_{2t}=(y_j^1-y_j^0, \ldots, y_j^{2t+1}-y_j^{2t}),
\]
\[
\overline{x}\T_{2t}=(x_j^1-x_j^0, \ldots, x_j^{2t+1}-x_j^{2t}).
\]
It has been shown in \cite{2013:Ye} that ${\bm \beta}_j$ can be computed as the kernel of the first defective Hankel matrices $\Gamma\{\overline{y}\T_{2t}\}$ and $\Gamma\{\overline{x}\T_{2t}\}$ for arbitrary initial conditions $y_0$ and $x_0$ (i.e., ${\bm \beta}_j$ can be calculated as the normalized kernel $ {\bm \beta}_j = \begin{bmatrix} {\bm \beta}_j^0 & {\bm \beta}_j^1 & \ldots & {\bm \beta}_j^{x_{ M_j}-1 } & 1 \end{bmatrix}\T $ of the first defective Hankel matrix $\Gamma\{\overline{y}\T_{2t}\}$), except a set of initial conditions with Lebesgue measure zero.

Next, we provide Theorem~\ref{lem:main} from \cite{themisCDC:2013}, in which it is stated that the exact average $\mu$ among the nodes in a strongly connected digraph can be distributively obtained in a finite number of steps.

\begin{theorem}[\hspace{-0.001cm \cite{themisCDC:2013}}]\label{lem:main}
Consider a strongly connected graph $\mathcal{G}(\mathcal{V}, \mathcal{E})$. Let $y_j^t$ and $x_j^t$ (for all $v_j \in \mathcal{V}$ and $t=0,1,2,\ldots$) be the result of the iterations \eqref{subeq:1} and \eqref{subeq:2}, where  $P = [p_{ji} ] \in \mathbb{R}_{+}^{n\times n}$ is any set of weights that adhere to the graph structure and form a primitive column stochastic weight matrix. Then, the solution to the average consensus can be distributively obtained in finite-time at each node $v_j$, by computing
\begin{align}
\mu_j \triangleq \lim_{t\rightarrow \infty} \frac{y_j^t}{x_j^t} = \frac{\phi_y (j)}{\phi_x (j)}=  \frac{y_{M_j}\T{\bm \beta}_j}{x_{M_j}\T{\bm \beta}_j} \; ,
\end{align}
where $\phi_y (j)$ and $\phi_x (j)$ are given by equations \eqref{phi:1} and \eqref{phi:2}, respectively and ${\bm \beta}_j$ is the vector of coefficients, as defined in \eqref{beta}.
\end{theorem}

Theorem~\ref{lem:main} states that the average consensus in a strongly connected digraph can be computed by the ratio of the final values computed for each of the iterations \eqref{subeq:1} with initial condition $y^0=y_0$ and iteration \eqref{subeq:2} with initial condition $x^0=\mathbb{1}$. Note that $x^0=\mathbb{1}$  does not belong into the Lebesgue measure zero set of matrix $P$ as defined in Proposition~\ref{lemma_christoforos}.

\subsection{$\max-$consensus algorithm}

The $\max-$ consensus algorithm is a simple algorithm for computing the maximum value in a distributed fashion \cite{2008:Cortes}. For any node $v_j\in \mathcal{V}$, the update rule is as follows:
\begin{align}
x_j^{t+1} = \max_{v_i\in \mathcal{N}_j^{-} \cup \{v_j\}}\{ x_i^t \}.
\end{align}
It has been shown (see, e.g., \cite{Silvia:CDC2013,Khatana:2020ACC}) that this algorithm converges to the maximum value among all nodes in a finite number of steps $s$, $s \leq D$.

\subsection{Finite-Time Distributed Termination  (FTDT)}
\label{sec_distributed_terminating_algorithm}

Charalambous and Hadjicostis \cite{charalambous2018stop} proposed a distributed termination mechanism in order to enhance an existing finite-time distributed algorithm to allow the nodes to agree when to terminate their iterations, provided they have all computed their exact average. More specifically, the proposed method is based on the fact that the finite-time consensus algorithm proposed in \cite{themisCDC:2013,themisTCNS:2015} allows nodes in the network running iterations \eqref{subeq:1} and \eqref{subeq:2}  to compute an upper bound of their eccentricity and use this information for deciding when to terminate the process. The procedure is as follows: 
\begin{list4}
\item Once iterations \eqref{subeq:1} and \eqref{subeq:2} are initiated, each node $v_j$ also initiates two counters $c_j$, $c_j[0]=0$, and $r_j$, $r_j[0]=0$. Counter $c_j$ increments by one at every time step, i.e., $c_j[k+1]= c_j[k]+1$. The way counter $r_j$ updates is described next.
\item Alongside iterations \eqref{subeq:1} and \eqref{subeq:2} a $\max$-consensus algorithm is initiated as well, given by
\begin{align}\label{eq:maxconsensus}
\theta_j[k+1] = \max_{v_i \in \mathcal{N}_j \cup \{v_j\}}\big\{ \max\{\theta_i[k],c_i[k]\} \big\}, 
\end{align}
with $\theta_j[0]=0$. Every time step $k$ for which $\theta_j[k+1]=\theta_j[k]$, counter $r_j$ increments by one, but if, however, at any step $k'$, $\theta_j[k'+1] \neq \theta_j[k']$, then $r_j$ is set to zero, i.e., 
\begin{align}\label{eq:rj}
r_j[k+1]=
\begin{cases}
0, & \text{if } \theta_j[k+1] \neq \theta_j[k], \\
r_j[k]+1, & \text{otherwise}.
\end{cases}
\end{align}
\item Once the square Hankel matrices $\Gamma \{ \overline{y}_{M_j}\T \}$ and $\Gamma \{\overline{x}_{M_j}\T \}$ for node $v_j$ lose rank, node $v_j$ saves the count of the counter $c_j$ at that time step, denoted by $k^o_j$, as $c^o_j$, i.e., $c^o_j\triangleq c_j[k^o_j]$, and it stops incrementing the counter, i.e., $\forall k'\geq k^o_j, c[k']=c_j[k^o_j]=c^o_j$. Note that $c^o_j=2(M_j+1)$. 
\item Node $v_j$ can terminate iterations \eqref{subeq:1} and \eqref{subeq:2} when $r_j$ reaches $c^o_j$.
\end{list4}
The main idea of this approach is that $2(M_j +1)$ serves as an upper bound on the maximum distance of any other node to node $v_j$. This quantity  is not known initially, but becomes known to node $v_j$ through the finite-time consensus algorithm \cite{themisTCNS:2015}, and it is used to decide whether all nodes have computed the average and, hence, the node can terminate the iterations.

\subsection{Standard ADMM Algorithm}

	The Standard ADMM algorithm solves the following problem
	\begin{equation}\label{standard_ADMM_obj}
	\begin{aligned}
	\min \, &f(x) + g(z) \\
	\text{s.t.} \, &Ax + Bz = c
	\end{aligned}
	\end{equation}
	for variables $ x \in \mathbb{R}^{p}, z \in \mathbb{R}^{m} $ with matrices $ A \in \mathbb{R}^{q\times p}, B \in \mathbb{R}^{q\times m} $ and vector $ c \in \mathbb{R}^{q} $. Note that $ p\in\mathbb{N}$ and $m \in\mathbb{N}$ represent the dimensions of prime variables. The augmented Lagrangian is
	\begin{equation}\label{augmented_Lagrangian}
	\begin{aligned}
	L_{\rho}(x,z,\lambda) =& f(x) + g(z)+ \lambda^{T}(Ax + Bz - c) \\&+ \frac{\rho}{2}\|Ax + Bz - c\|^{2},
	\end{aligned}
	\end{equation}
	where $ \lambda $ is the Lagrange multiplier and $ \rho >0 $ is a penalty parameter. In ADMM, the primary variables $ x, z $ and the Lagrange multiplier $ \lambda $ are updated as follows: starting from some initial vector $\begin{bmatrix}x^0 & z^0 & \lambda^0 \end{bmatrix}\T $, at each optimization iteration $ k $, 
	\begin{align}
	x^{k+1}=&  \operatorname*{argmin}_x L_{\rho}(x,z^k,\lambda^k), \label{admm_x}\\
	z^{k+1} =& \operatorname*{argmin}_z L_{\rho}(x^{k+1},z,\lambda^k), \label{admm_z}\\
	\lambda^{k+1} =& \lambda^k + \rho (Ax^{k+1} + Bz^{k+1} - c) \label{admm_lamda}.
	\end{align}
	The step-size in the Lagrange multiplier update is the same as the augmented Lagrangian function parameter $ \rho $.

%
%
%
%
\section{Problem Formulation}\label{sec:formulation}

In this work, we consider a strongly connected digraph $\mathcal{G}(\mathcal{V}, \mathcal{E})$ in which each node $v_j \in\mathcal{V}$ is endowed with a scalar cost function $f_i : \mathbb{R}^p \mapsto \mathbb{R}$ assumed to be known to the node only. We assume that each node $v_j$ has knowledge of the number of its out-going links, $\mathcal{D}^{+}_{j}$, and has access to local information only via its communication with the in-neighboring nodes, $\mathcal{N}^{-}_{j}$. The only global information available to all the nodes in the network is given in Assumption~\ref{assup_graph}.

\begin{assumption}\label{assup_graph}
Each node $v_j \in\mathcal{V}$ knows an upper bound on the number of nodes in the network $n'$ (i.e., $n' \geq n$).
\end{assumption}

While Assumption~\ref{assup_graph} is limiting, there exist distributed methods for computing the size of the network; see, for example, \cite{2012:Allerton_Shames}.



The problem is to design a discrete-time coordination algorithm that allows every node $v_j$ in a \emph{digraph} to distributively solve the following optimization problem:
\begin{equation}\label{problem_initial}
\operatorname*{argmin}_{x\in \mathbb{R}^{p}}  \sum_{i=1}^{n} f_i(x),
\end{equation}
where $ x\in \mathbb{R}^{p} $ is a global optimization variable (or a common decision variable). In order to distributively solve the previous problem and  to enjoy the structure ADMM scheme at the same time, a separate decision variable $ x_i $ for node $v_i$ is introduced and the constraint $ x_i=x_j $ is imposed to guarantee that the node decision variables are equal\footnote{This step is quite standard in distributed optimization.}. In other words, problem~\eqref{problem_initial} is reformulated as
\begin{equation}\label{problem_reformulated}
\begin{aligned}
\min \, &\sum_{i=1}^{n} f_i(x_i), \\
\text{s.t.}  \, &x_i = x_j, \forall v_i, v_j\in\mathcal{V}.
\end{aligned}
\end{equation}
Define a closed nonempty convex set $ \mathcal{C} $ as
\begin{equation}\label{setC}
\mathcal{C} = \left\{\begin{bmatrix} x_1\T & x_2\T & \ldots & x_n\T\end{bmatrix}\T \in \mathbb{R}^{np} \, :\,  x_i=x_j \right\} .
\end{equation}
By denoting $ X \coloneqq \begin{bmatrix} x_1\T & x_2\T & \ldots & x_n\T\end{bmatrix}\T $ and making variable $ z \in \mathbb{R}^{np} $ as a copy of vector $ X $, problem~\eqref{problem_reformulated} becomes
\begin{equation}\label{problem_reformulated2}
\begin{aligned}
\min \, & \sum_{i=1}^{n} f_i(x_i), \\
\text{s.t.}  \, & X = z, \, z \in \mathcal{C}. 
\end{aligned}
\end{equation}
Then, take $ g $ as the indicator function of set $ \mathcal{C}  $, and define $ g(z) $ as 
\begin{equation}\label{indicator_function}
g(z) =
\left\{ 
\begin{array}{l}
\begin{aligned}
&0,\quad \text{if} \, z \in \mathcal{C},\\
&\infty, \, \text{otherwise}.
\end{aligned}
\end{array}
\right. 
\end{equation}
Finally, problem~\eqref{problem_reformulated2} is transformed to 
\begin{equation}\label{objective_function}
\begin{aligned}
\min \, & \sum_{i=1}^{n} f_i(x_i) + g(z), \\
\text{s.t.} \, & X - z = 0. 
\end{aligned}
\end{equation}

For notational convenience, denote $ F(X) \coloneqq \sum_{i=1}^{n} f_i(x_i) $. Thus, denote the Lagrangian function as
\begin{equation}\label{Lagrangian}
L(X,z,\lambda) = F(X) + g(z) + \lambda\T(X - z ),
\end{equation}
where $ \lambda $ in $ \mathbb{R}^{np} $ is the Lagrange multiplier associated with the constraint $ X - z =0 $. Then, the following standard assumptions are required for the optimization problem.
\begin{assumption}\label{assup_convex}
	Each cost function $ f_i : \mathbb{R}^{p} \rightarrow \mathbb{R} \cup \{+\infty\} $ is closed, proper and convex.
\end{assumption}
\begin{assumption}\label{assup_saddel_point}
	The Lagrangian $ L(X,z,\lambda) $ has a saddle point, i.e., there exists a solution $ (X^{*},z^{*},\lambda^{*}) $, for which 
	\begin{equation}\label{saddle_point}
	L(X^{*},z^{*}, \lambda)\le L (X^{*},z^{*},\lambda^{*})\le L(X,z,\lambda^{*})
	\end{equation}
	holds for all $ X $ in $ \mathbb{R}^{np} $, $ z $ in $ \mathbb{R}^{np} $ and $ \lambda $ in $ \mathbb{R}^{np} $.
\end{assumption}

Assumption~\ref{assup_convex} allows $ f_i $ to be non-differentiable~\cite{boyd2011distributed}. By Assumptions~\ref{assup_convex}-\ref{assup_saddel_point} and based on the definition of $ g(z) $ in~\eqref{indicator_function}, $ L(X,z,\lambda^{*}) $ is convex in $ (X,z) $ and $ (X^{*},z^{*}) $ is a solution to problem~\eqref{objective_function}~\cite{boyd2011distributed,wei2012distributed}.


%
%
%
%
\section{Main results}\label{sec_mainResult}
\label{sec:results}

At iteration $ k $, the corresponding augmented Lagrangian of optimization problem~\eqref{objective_function} is written as
\begin{align}\label{augmented_Lagrangian2}
L_{\rho}&(X^{k},z^{k},\lambda^{k}) \\
=& \sum_{i=1}^{n} f_i(x_i^{k}) + g(z^{k}) + {\lambda^{k}}\T(X^{k} - z^{k} )+ \frac{\rho}{2}\|X^{k} - z^{k} \|^{2} \nonumber \\
=&  \sum_{i=1}^{n}\left( f_i(x_i^{k}) + {\lambda_i^{k}}\T(x_i^{k} - z_i^{k}) + \frac{\rho}{2}\|x_i^{k} - z_i^{k} \|^{2} \right) +g(z^{k}) , \nonumber
\end{align}
where $ z_i\in \mathbb{R}^{p}  $ is the $ i-$th element of vector $ z $. By ignoring terms which are independent of the minimization variables (i.e., $ x_i, z $), for each node $ v_i $, the standard ADMM updates \eqref{admm_x}-\eqref{admm_lamda} change to the following format:
\begin{align}
x_i^{k+1} =&  \operatorname*{argmin}_{x_i} f_i(x_i) + {\lambda_i^{k}}\T x_i + \frac{\rho}{2}\|x_i - z_i^{k} \|^{2}, \label{dadmm_x}\\
z^{k+1} =& \operatorname*{argmin}_z g(z) + {\lambda^{k}}\T(X^{k+1} - z )  + \frac{\rho}{2}\|X^{k+1} - z \|^{2}\nonumber \\
=& \operatorname*{argmin}_z g(z) +\frac{\rho}{2}\|X^{k+1} - z + \frac{1}{\rho}\lambda^{k} \|^{2}, \label{dadmm_z}\\
\lambda_i^{k+1} =& \lambda_i^{k} + \rho (x_i^{k+1} - z_i^{k+1}) \label{dadmm_lamda},
\end{align}
where the last term in~\eqref{dadmm_z} comes from the identity $ 2a^Tb + b^2=(a+b)^2-a^2 $ with $ a = \lambda^{k}/\rho $ and $ b = X^{k+1} - z $.

Update~\eqref{dadmm_x} for $ x_i^{k+1} $ can be solved by a classical method, e.g., the proximity operator~\cite[Section 4]{boyd2011distributed}. Update~\eqref{dadmm_lamda} for the dual variable $ \lambda_i^{k+1} $  can be implemented trivially by node $ v_i $. 
Note that both updates can be done independently by node $ v_i $. 
Since $ g $ is the indicator function of the closed nonempty convex set $ \mathcal{C}  $,  update~\eqref{dadmm_z} for $ z^{k+1} $ becomes 
$$ 
z^{k+1} = \Pi_{\mathcal{C}}(X^{k+1}+ \lambda^{k}/\rho) , 
$$
where $ \Pi_{\mathcal{C}} $ denotes the projection (in the Euclidean norm) onto $ \mathcal{C} $.
Intuitively, from \eqref{dadmm_z} and the definition of $ g(z) $ in~\eqref{indicator_function}, one can see that the elements of $ z $ (i.e., $ z_1, z_2, \ldots, z_n $)  should go into $ \mathcal{C} $ in finite time. If not, one will have $  g(z)  = \infty $ and update~\eqref{dadmm_z} will never be finished. Then, from the definition of $ \mathcal{C} $ in~\eqref{setC}, one can see that $ z $ going into $ \mathcal{C} $ means $  z_1= z_2= \ldots = z_n $, which is in the mathematical format of consensus. Therefore, if each node $v_i\in\mathcal{V}$ can have $z_i$ reach $ \frac{1}{n} \sum_{i=1}^{n}z_i(0) $ in a finite number of steps, with $ z_i(0) = x_i^{k+1} + \lambda_i^{k}/\rho  $, then the update can be completed. Therefore, update~\eqref{dadmm_z} reduces to a finite time consensus problem. For this reason, we adopt the finite time exact ratio consensus (FTERC) algorithm for digraphs, as introduced in Section~\ref{FTERC}.

\subsection{FTERC given network size upper bound~\cite{ECC2021:Wei-Themis}}
\label{sec:D-ADMM-FTERC}

Herein, we describe the algorithm we proposed in~\cite{ECC2021:Wei-Themis}. For this algorithm, we assume that all nodes are aware of an upper bound of the size of the network $n'$ (i.e., $n' \geq n$ and $n'$ is known to all nodes), the augmented Lagrangian function parameter $\rho$, and the ADMM maximum  optimization step $k_{\max}$. The optimization consists of the following steps:
\begin{list4}
\item[1)] At the first optimization step, node $v_i\in\mathcal{V}$ computes  $x_i^{1}$ using~\eqref{dadmm_x}, computes $z_i^{1}$ via FTERC which runs for $2n'$ iterations. By that time, it is guaranteed that each node has computed their final value $z_i^{1}$, which requires computing $ {\bm \beta}_i $ and as a consequence $M_i$ is determined. Then, using
 $x_i^{1}$ and $z_i^{1}$, it computes $\lambda_i^{1}$ using~\eqref{dadmm_lamda}.
\item[2)] At the second optimization step, node $v_i\in\mathcal{V}$, computes  $x_i^{2}$ using~\eqref{dadmm_x}, runs ratio consensus~\eqref{eq:1} for $n'$ iterations and computes $z_i^{2}$ with the same $ {\bm \beta}_i $ computed at the first optimization step (i.e., there is no need to compute the defective Hankel matrices again). At the same time it runs a $\max-$consensus algorithm with initial condition $x_i^0  = M_i+1$. Note that $M_i+1 < n \leq n'$ and that the $\max-$consensus algorithm converges in $s$ iterations ($s\leq D \leq n-1<n'$). Hence, at this step node $v_i$, not only computes $z_i^{2}$, but also the maximum number of iterations needed $t_{\max} \coloneqq M_{\max}+1$ by each node $v_i\in\mathcal{V}$ for every optimization step $k$ to compute their $z_i^{k+1}$. Again, using $x_i^{2}$ and $z_i^{2}$, it computes $\lambda_i^{2}$ using~\eqref{dadmm_lamda}.
\item[3)] At every optimization step thereafter, each node $v_i\in\mathcal{V}$ computes  $x_i^{k+1}$ using~\eqref{dadmm_x}, computes $z_i^{k+1}$ via ratio consensus~\eqref{eq:1} with the same $ {\bm \beta}_i $ which runs for $t_{\max}$ iterations.
\item[4)] The ADMM algorithm terminates once the stopping criterion\footnote{Primal and dual feasibility conditions in~\cite{boyd2011distributed}.} is satisfied or the maximum number of optimization steps, $k_{\max}$ is reached. 
\end{list4}

The algorithm guarantees that the number of iterations needed at every optimization step $k$, $k\geq 2$ is the \emph{minimum} (see properties of FTERC) and that the solution at every step is the \emph{exact} optimal. Fig.~\ref{FTERC_structure} shows the number of iterations needed at every optimization step. 
\begin{figure}[h!]
	\centering
\includegraphics[width=\columnwidth]{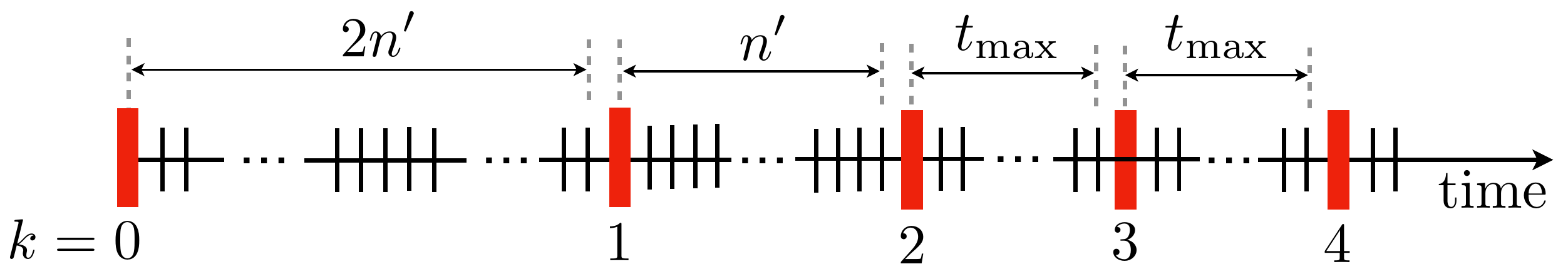}
\caption{The FTERC is terminated after $2n'$ iterations in the first step of the ADMM optimization ($ n' $ is an upper bound of $ n $ known to all nodes). During the first optimization step, each node $v_j$ computes $M_j$. During the second step, they run a max-consensus algorithm (which converges in $s$ ($s\leq D$) iterations) and they determine $M_{\max}$. The second optimization step is terminated after $n'$ iterations. Thereafter, each optimization step is changed to $t_{\max} \coloneqq M_{\max}+1$ iterations.}
\label{FTERC_structure}
\end{figure}

Algorithm~\ref{algorithm:1} provides D-ADMM-FTERC, the distributed ADMM algorithm proposed in~\cite{ECC2021:Wei-Themis} for solving optimization problem~\eqref{objective_function}.

\begin{algorithm}[t!]
	\caption{D-ADMM-FTERC.}
	\begin{algorithmic}[1]
		\STATE \textbf{Input:} $\mathcal{G}(\mathcal{V}, \mathcal{E})$, $\rho >0$, $n'$ (upper bound on $n$), $ k_{\max} $ (ADMM maximum number of iterations)
		\STATE \textbf{Data:} Node $ v_i \in \mathcal{V}$ sets $ x_i^0, z_i^0, \lambda_i^0 $ randomly, and $k=0$
		\newline
		\STATE Node $ v_i \in \mathcal{V}$ does the following:
		\WHILE {$ k \le k_{\max} $ }
		\STATE Compute $x_i^{k+1}$ using Eq.~\eqref{dadmm_x}
		\IF {$ k = 0 $} 
		\STATE Compute $z_i^{1}$ via FTERC which runs for $2n'$ steps, and determine $M_i$ and $ {\bm \beta}_i $
		\ELSIF {$ k = 1 $} 
		\STATE Run $\max-$consensus and ratio consensus~\eqref{eq:1}: determine $M_{\max}$ (via $\max-$consensus),  compute $z_i^{2}$ with the same $ {\bm \beta}_i $ and terminate iterations after $n'$ steps
		\ELSE {}
		\STATE Compute $z_i^{k+1}$ via ratio consensus~\eqref{eq:1} with the same $ {\bm \beta}_i $ (runs for $t_{\max}$)
		\ENDIF
		\STATE Compute $\lambda_i^{k+1}$ using Eq.~\eqref{dadmm_lamda}
		\IF{ADMM stopping criterion is satisfied}
		\STATE Stop D-ADMM-FTERC
		\ENDIF
		\STATE $k\leftarrow k+1$ 
		\ENDWHILE
	\end{algorithmic}
	\label{algorithm:1}
\end{algorithm}

\subsection{Fully distributed FTERC}

In this subsection, we propose a \emph{fully-distributed} FTERC algorithm {using FTDT in Section~\ref{sec_distributed_terminating_algorithm}}, in which no information about the network size is known. The changes with respect to D-ADMM-FTERC described in Section~\ref{sec:D-ADMM-FTERC} are steps 1) and 2).  More specifically, step 2) is completely omitted, whereas step 1) is replaced by the following step: At the first optimization step, node $v_i\in\mathcal{V}$ computes  
\begin{list4}
\item $x_i^{1}$ using~\eqref{dadmm_x},  
\item $z_i^{1}$ via {FTDT} which runs for $ t_1 \coloneqq 4(M_{\max}+1) -1 $ iterations \emph{for all nodes} (as shown in Fig.~\ref{FTERC_structure_2}). During the first $t_1$ iterations, node $v_i$ will terminate at step $t_{0,i}\leq t_1$, given by
\begin{align}\label{eq:t0i}
t_{0,i} \coloneqq \max_{v_j\in \mathcal{N}}\{2(M_j+1)\} + 2(M_i+1)-1 ,
\end{align}
as it is the case for FTERC in Section~\ref{sec:D-ADMM-FTERC}. We observe that, by $t_{0,i}$ node $v_i$ \emph{(i)} has computed their final value $z_i^{1}$, which requires computing $ {\bm \beta}_i $, and as a consequence $M_i$ is determined, and \emph{(ii)} $  M_{\max} \coloneqq \max_{v_i\in \mathcal{N}}M_i $ can be determined by \eqref{eq:t0i}
\begin{align}
M_{\max} = \frac{t_{0,i}-2M_i-1}{2} -1.
\end{align}
As a consequence, node will terminate its iteration at $t_{0,i}$ and it will know that by $t_1$ all the nodes have terminated their iterations.
\item $\lambda_i^{1}$ using $x_i^{1}$ and $z_i^{1}$ in~\eqref{dadmm_lamda}.
\end{list4}
Note that this process is for synchronizing of all nodes to start the next ADMM iteration simultaneously. Thereafter, each ADMM optimization step is changed to $t_{\max}$. Note also that, as it is the case with Algorithm~\ref{algorithm:1}, since the nodes have computed $M_{\max}$ already, they will be able to determine $t_{\max}$ and this is what they do from the optimization step 1 onwards (something that in Algorithm~\ref{algorithm:1} it starts from optimization step 2 onwards; see Fig~\ref{FTERC_structure_2}.  
\begin{figure}[h!]
	\centering
	\includegraphics[width=\columnwidth]{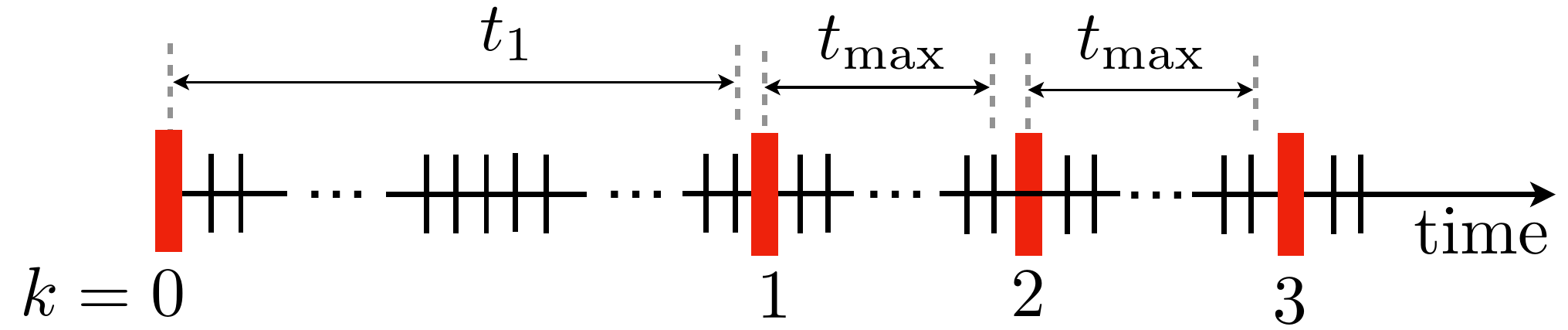}
	\caption{The finite-time consensus algorithm is terminated after $t_1$ iterations in the first step of the ADMM optimization. During the first optimization step, each node $v_i$ computes $M_i$. Note that inside this process, each node $v_i$ can compute the value of $M_i$ using  FTDT in Section~\ref{sec_distributed_terminating_algorithm} at step $ 2(M_i+1) -1 $; then, at step $ t_{0,i} $, the value of $ M_{\max}$ can be calculated. After step $ t_{0,i} $, each node $v_i$ stops running the ratio consensus process until everybody terminates at step $ t_1$. }
	\label{FTERC_structure_2}
\end{figure}

We stress that the algorithm guarantees that the number of iterations needed at every optimization step $k$, $k\geq 1$ is the \emph{minimum} (see properties of {FTDT}) and that the solution at every step is the \emph{exact} optimal.

We now formally describe our algorithm, herein called \emph{Algorithm~\ref{algorithm:2}},  in which the nodes distributively solve optimization problem~\eqref{objective_function}. 

\begin{algorithm}[t!]
	\caption{FD-ADMM-FTDT.}
	\begin{algorithmic}[1]
		\STATE \textbf{Input:} $\mathcal{G}(\mathcal{V}, \mathcal{E})$, $\rho >0$, $ k_{\max} $ (ADMM maximum number of iterations)
		\STATE {\textbf{Initialization:}} Node $ v_i \in \mathcal{V}$ sets $ x_i^0, z_i^0, \lambda_i^0 $ randomly, and $k=0$
		\newline
		\STATE Node $ v_i \in \mathcal{V}$ does the following:
		\WHILE {$ k \le k_{\max} $ }
		\STATE Compute $x_i^{k+1}$ using Eq.~\eqref{dadmm_x}
		\IF {$ k = 0 $} 
		\STATE Compute $z_i^{1}$ via {FTDT} which runs for $t_1$ steps, and determine $M_i$, $ {\bm \beta}_i $ and $ M_{\max} $ (or $ t_{\max} $).
		\ELSE {}
		\STATE Compute $z_i^{k+1}$ via ratio consensus~\eqref{eq:1} with the same $ {\bm \beta}_i $ (runs for $t_{\max}$)
		\ENDIF
		\STATE Compute $\lambda_i^{k+1}$ using Eq.~\eqref{dadmm_lamda}
		\IF{ADMM stopping criterion is satisfied}
		\STATE Stop FD-ADMM-FTDT
		\ENDIF
		\STATE $k\leftarrow k+1$ 
		\ENDWHILE
	\end{algorithmic}
	\label{algorithm:2}
\end{algorithm}

\begin{remark}
The main difference between Algorithm~\ref{algorithm:1} and \ref{algorithm:2} is that Algorithm~\ref{algorithm:2} does not require the information of the upper bound of the size of the network $n'$. This is especially helpful and significant when dealing with large-scale multi-node optimization problems.
\end{remark}

\begin{remark}
For {both} algorithms (Algorithms~\ref{algorithm:1} and~\ref{algorithm:2}), the exact finite-time iteration number $ t_{\max} $ is determined theoretically. On the contrary, for  finite-time $ \epsilon $ consensus in~\cite{khatana2020d}, the iteration number cannot be decided by theory and also varies depending on the value of $ \epsilon $.
\end{remark}

%
%
%
%
\section{Convergence analysis}\label{sec:convergence_analysis}
In this section, the following well-know identity is used frequently:
\begin{align}
(a_1&-a_2)\T(a_3-a_4) =\frac{1}{2}(\|a_1-a_4\|^2- \|a_1-a_3\|^2 ) \label{equality_law}\\
&+\frac{1}{2}(\|a_2-a_3\|^2- \|a_2-a_4\|^2 ), \forall a_1,a_2,a_3,a_4 \in \mathbb{R}^p.\nonumber
\end{align}
\subsection{$O(1/k)$ converge rate}

In this section, the $ O(1/k) $ convergence rate of our proposed D-ADMM-FTERC{/FD-ADMM-FTDT} algorithm{s} will be presented. The analysis is inspired by the analyses in \cite{wei2012distributed} and \cite{khatana2020d}. Authors in \cite{wei2012distributed} analyzed the D-ADMM based on the assumption of an underlying undirected graph. A D-ADMM approach for directed graphs is proposed in \cite{khatana2020d} based on a finite-time ``approximate'' consensus method, which means the resulted solution to problem~\eqref{objective_function} will not be optimal, but close to the optimal solution $ (X^{*},z^{*}) $. By using the FTERC{/FTDT} method{s} presented in the previous section, 
we will prove our D-ADMM-FTERC{/FD-ADMM-FTDT} algorithm{s} for the digraph being able 
to get the exact optimal solution $ (X^{*},z^{*}) $ in the following theorem.
\begin{theorem}\label{theorem_convergence}
	Let $ \{X^k, z^k, \lambda^k\} $ be the iterates from our Algorithms~\ref{algorithm:1} and~\ref{algorithm:2} for problem~\eqref{objective_function}, where $ X^k = [{x_1^{k}}\T,{x_2^{k}}\T, \ldots, {x_n^{k}}\T]\T $ and $ \lambda^k = [{\lambda_1^{k}}\T,{\lambda_2^{k}}\T, \ldots, {\lambda_n^{k}}\T]\T $. Let $ \bar{X}^k = \frac{1}{k} \sum_{s=0}^{k-1}X^{s+1}, \bar{z}^k = \frac{1}{k} \sum_{s=0}^{k-1}z^{s+1} $ be respectively the ergodic average of $ X^k $.
	Considering a strongly connected communication graph, 
	under Assumptions \ref{assup_graph}-\ref{assup_saddel_point}, the following relationship holds for any iteration $ k $ as
	\begin{equation}\label{convergence_relationship}
	\begin{aligned}
	0&\le L(\bar X^{k},\bar z^{k}, \lambda^{*})- L ( X^{*},z^{*},\lambda^{*}) \\
	&\le \frac{1}{k}\left( \frac{1}{2\rho} \|\lambda^{*}-\lambda^{0}\|^2 + \frac{\rho}{2}\|X^{*}-z^{0}\|^2\right).
	\end{aligned}
	\end{equation}
\end{theorem}
\begin{proof}
	From the second inequality of the saddle point of Lagrangian function~\eqref{saddle_point}, the first inequality in~\eqref{convergence_relationship} can be proved directly.
	
	We now prove the second inequality in~\eqref{convergence_relationship}. For each node $ v_i $, since $ x_i^{k+1} $ minimizes $ L_{\rho}(x,z^k, \lambda^k) $ in~\eqref{dadmm_x}, by the optimal condition, we have 
	\begin{equation}
	\begin{aligned}
	(x- x_i^{k+1})\T[h_i(x_i^{k+1}) + \lambda_i^k + \rho (x_i^{k+1} - z_i^k) ] \ge 0,
	\end{aligned}
	\end{equation}
	where $ h_i(x_i^{k+1}) $ is the sub-gradient of $ f_i $ at $ x_i^{k+1} $. By integrating $ x_i^{k+1} = (\lambda_i^{k+1}-\lambda_i^k)/\rho + z_i^{k+1} $ from~\eqref{dadmm_lamda} into the above inequality, we have
	\begin{equation}
	(x- x_i^{k+1})\T[h_i(x_i^{k+1}) + \lambda_i^{k+1} + \rho (z_i^{k+1} - z_i^k) ] \ge 0.
	\end{equation}
	The compact mathematical format of the above $ n $ inequalities can be written as
	\begin{equation}\label{optimal_condition_X}
	(X- X^{k+1})\T[\bar h(X^{k+1}) + \lambda^{k+1} + \rho (z^{k+1} - z^k) ] \ge 0,
	\end{equation}
	where $ \bar h(X^{k+1}) = [h_1\T(x_1^{k+1}), \ldots, h_n\T(x_n^{k+1})]\T $. Since $ z^{k+1} $ minimizes $ L_{\rho}(x^{k+1},z, \lambda^k) $ in~\eqref{dadmm_z}, similarly, we have
	\begin{equation}\label{optimal_condition_z}
	\begin{aligned}
	(z- &z^{k+1})\T[\bar g(z^{k+1}) - \lambda^{k} - \rho (X^{k+1} - z^{k+1}) ]\\
	&=(z- z^{k+1})\T (\bar g(z^{k+1}) - \lambda^{k+1}) \ge 0
	\end{aligned}
	\end{equation}
	for all $ z \in \mathcal{C} $,
	where $ \bar g(z^{k+1}) $ is the sub-gradient of $ g $ at $ z^{k+1} $. As both $ F $ and $ g $ are convex, by utilizing the sub-gradient inequality, we get 
	\begin{align}
	F(&X^{k+1}) - F(X) + g(z^{k+1}) - g(z) \nonumber\\
	\le& -(X- X^{k+1})\T\bar h(X^{k+1}) - (z- z^{k+1})\T\bar g(z^{k+1})\nonumber \\
	\le& \lambda^{(k+1)T} [X- X^{k+1} -(z- z^{k+1})] \nonumber\\
	&+ \rho (X- X^{k+1})\T(z^{k+1} - z^k),\label{proof1}
	\end{align}
	where the last inequality comes from \eqref{optimal_condition_X} and \eqref{optimal_condition_z}. 
	Due to feasibility of the optimal solution $ (X^{*}, z^{*}) $, we obtain $ X^{*}- z^{*}=0 $.
	By setting $ X=X^{*}, z=z^{*}  $, \eqref{proof1} becomes
	\begin{align}
	F(&X^{k+1}) - F(X^{*}) + g(z^{k+1}) - g(z^{*})\label{proof2}\\
	\le& \lambda^{{(k+1)}\T} (z^{k+1}- X^{k+1} ) + \rho (X^{*}- X^{k+1})\T(z^{k+1} - z^k).\nonumber
	\end{align}
	Adding the term $ {\lambda^{*}}\T( X^{k+1}-z^{k+1} ) $ to both sides of \eqref{proof2}, 
	{\small
	\begin{equation}\label{proof3}
	\begin{aligned}
	F(&X^{k+1}) - F(X^{*}) + g(z^{k+1}) - g(z^{*}) + {\lambda^{*}}\T( X^{k+1}-z^{k+1})\\
	\le& (\lambda^{*}-\lambda^{k+1})\T (X^{k+1}-z^{k+1} ) + \rho (X^{*}- X^{k+1})\T(z^{k+1} - z^k)
	\\
	\le&\frac{1}{\rho}(\lambda^{*}-\lambda^{k+1})\T(\lambda^{k+1}-\lambda^{k} ) + \rho (X^{*}- X^{k+1})\T(z^{k+1} - z^k),
	\end{aligned}
	\end{equation}}where the last equality is calculated from~\eqref{dadmm_lamda}. 
	Then, by using equality~\eqref{equality_law}, \eqref{proof3} changes to
	\begin{align}
	F(&X^{k+1}) - F(X^{*}) + g(z^{k+1}) - g(z^{*}) + {\lambda^{*}}\T( X^{k+1}-z^{k+1}) \nonumber\\
	\le&\frac{1}{2\rho}(\|\lambda^{*}-\lambda^{k}\|^2- \|\lambda^{*}-\lambda^{k+1}\|^2 
	+\|\lambda^{k+1}-\lambda^{k+1}\|^2\nonumber\\&- \|\lambda^{k+1}-\lambda^{k}\|^2 ) + \frac{\rho}{2} (\|X^{*}-z^k\|^2- \|X^{*}-z^{k+1}\|^2\nonumber\\& +\| X^{k+1}-z^{k+1}\|^2- \| X^{k+1}-z^k\|^2 )\nonumber\\
	\le & \frac{1}{2\rho}(\|\lambda^{*}-\lambda^{k}\|^2- \|\lambda^{*}-\lambda^{k+1}\|^2 )\nonumber\\&+\frac{\rho}{2} (\|X^{*}-z^k\|^2- \|X^{*}-z^{k+1}\|^2),\label{proof4}
	\end{align}
	where the last inequality comes from using~\eqref{dadmm_lamda} and dropping the negative term $ -\frac{\rho}{2}\| X^{k+1}-z^k\|^2 $. Now, by using $ s \coloneqq k $, we change~\eqref{proof4} to another format as
	\begin{align}
	F(&X^{s+1}) - F(X^{*}) + g(z^{s+1}) - g(z^{*}) + {\lambda^{*}}\T( X^{s+1}-z^{s+1})\nonumber\\
	\le & \frac{1}{2\rho}(\|\lambda^{*}-\lambda^{s}\|^2- \|\lambda^{*}-\lambda^{s+1}\|^2 )\nonumber\\&+\frac{\rho}{2} (\|X^{*}-z^s\|^2- \|X^{*}-z^{s+1}\|^2),\label{proof5}
	\end{align}
	which holds true for all $ s $. By summing~\eqref{proof5} over $ s=0,1,\ldots, k-1 $ and after telescoping calculation, we have
	\begin{equation}\label{proof6}
	\begin{aligned}
	\sum_{s=0}^{k-1}F(&X^{s+1}) - kF(X^{*}) + \sum_{s=0}^{k-1}g(z^{s+1}) - kg(z^{*}) \\&+ {\lambda^{*}}\T\sum_{s=0}^{k-1}( X^{s+1}-z^{s+1})\\
	\le & \frac{1}{2\rho}(\|\lambda^{*}-\lambda^{0}\|^2- \|\lambda^{*}-\lambda^{k}\|^2 )\\&+\frac{\rho}{2} (\|X^{*}-z^0\|^2- \|X^{*}-z^{k}\|^2).
	\end{aligned}
	\end{equation}
	Due to the convexity of both $ F $ and $ g $, we get $ kF(\bar X^k) \le \sum_{s=0}^{k-1}F(X^{s+1}) $ and $ kg(\bar z^k) \le \sum_{s=0}^{k-1}g(z^{s+1}) $. Thus, utilizing the definition of $ \bar X^k, \bar z^k $ and dropping the negative terms, we have
	\begin{align}
	kF(\bar X^k) &- kF(X^{*}) + kg(\bar z^k) - kg(z^{*}) + {\lambda^{*}}\T( k\bar X^k-k\bar z^k)\nonumber\\
	\le & \frac{1}{2\rho}\|\lambda^{*}-\lambda^{0}\|^2+\frac{\rho}{2} \|X^{*}-z^0\|^2.\label{proof7}
	\end{align}
	Based on $ X^{*}- z^{*}=0 $, \eqref{proof7} combined with the definition of Lagrangian function [cf.~\eqref{Lagrangian}] prove~\eqref{convergence_relationship}.
\end{proof}
\begin{remark}\label{remark_comparison}
	The convergence proof here is basically different from the ones in~\cite{wei2012distributed} and \cite{khatana2020d} as the investigated problems are different. Specifically, in~\cite{wei2012distributed}, the proposed D-ADMM can be only applied to nodes with undirected graphs as the constraint $ AX=0 $ is needed to minimize the objective function~\eqref{problem_reformulated}, where the matrix $ A $ is related to the communication graph structure which must be undirected. To apply D-ADMM for digraphs, authors in \cite{khatana2020d} proposed a different constraint which is $ \| x_i-x_j \|\le \varepsilon, \varepsilon>0 $ with the value of $ \varepsilon  $ predefined. Note that the above constraint will inevitably lead to a sub-optimal solution, which is close to the optimal but not the exact optimal as we can see from the comparisons in Section~\ref{sec:examples}.  In this paper, we propose the constraint $ x_i=x_j $ to guarantee the solution is optimal by using the FTERC{/FTDT} algorithm{s}.
\end{remark}


\subsection{Linear convergence rate}

\begin{assumption}\label{assump_Lipschitz_differentiable}
	Function $ f_i(\cdot) $ in~\eqref{problem_initial} is continuously differentiable.
\end{assumption}
\begin{assumption}\label{assump_strongly_convex}
	$ F(X) \coloneqq \sum_{i=1}^{n} f_i(x_i) $ is a strongly convex with $ \mu > 0 $, i.e.,
	\begin{equation}\label{important_inequality0}
	\begin{aligned}
	\langle \nabla \bar{F}(X)-\nabla \bar{F}(Y), X-Y \rangle \ge \mu\| X-Y\|^2, \forall X, Y \in \mathcal{C},
	\end{aligned}
	\end{equation}
	where $ \nabla \bar{F}(X) \coloneqq [\nabla f_1\T (x_1), \ldots, \nabla f_n\T (x_n)]\T $.
\end{assumption}
Different from Assumption~\ref{assup_convex}, assumption~\ref{assump_Lipschitz_differentiable} requires $ f_i(\cdot) $ to be differentiable.
Assumption~\ref{assump_strongly_convex} does not require each $ f_i(\cdot) $ to be strongly convex while the work of~\cite{makhdoumi2017convergence} does.  

\begin{theorem}
	Let $ \{X^k, z^k, \lambda^k\} $ be the iterates from Algorithms~\ref{algorithm:1} and \ref{algorithm:2} for problem~\eqref{objective_function}, where $ X^k = [{x_1^{k}}\T,{x_2^{k}}\T, \ldots, {x_n^{k}}\T]\T $ and $ \lambda^k = [{\lambda_1^{k}}\T,{\lambda_2^{k}}\T, \ldots, {\lambda_n^{k}}\T]\T $. 
	Considering a strongly connected communication graph, 
	under Assumptions \ref{assup_graph}-\ref{assump_strongly_convex}, $ X^k $
	 converges approximately R-linearly.
\end{theorem}
\begin{proof}
	From Assumption
	 \ref{assump_Lipschitz_differentiable}, $ f_i(x_i) $ is continuously differentiable for all $ x_i \in \text{dom} f_i $. By the first order optimality condition, for the optimal solution $ (X^{*},z^{*},\lambda^{*}) $, i.e., $  x_i^{*}-z_i^{*}=0 $, from Eqs.~\eqref{dadmm_x}, \eqref{dadmm_z} and \eqref{dadmm_lamda}, we have
	\begin{align}
	0 = \nabla f_i(x_i^{*}) + \lambda_i^{*}. \label{optimal_gradient}
	\end{align}
	Similarly, for solution $ (X^{k+1},z^{k+1},\lambda^{k+1}) $, we get
	\begin{align}
	0 =& \nabla f_i(x_i^{k+1}) + \lambda_i^{k} + \rho (x_i^{k+1} - z_i^{k}) \nonumber\\
	=& \nabla f_i(x_i^{k+1}) + \lambda_i^{k+1} + \rho (z_i^{k+1} - z_i^{k}),\label{optimal_k2}
	\end{align}

	From \eqref{important_inequality0}, by assigning $ X \coloneqq X^{k+1}, Y \coloneqq X^{*} $, we get
	\begin{equation}\label{important_inequality}
	\begin{aligned}
	\langle \nabla \bar{F}(X^{k+1})-\nabla \bar{F}(X^{*}), X^{k+1}-X^{*} \rangle \ge \mu\| X^{k+1}-X^{*}\|^2.
	\end{aligned}
	\end{equation}
	Based on the optimal condition theory, we have $  \nabla \bar{F}(X^{*})\T (X-X^{*}) \ge 0, \forall X \in \mathcal{C}$, i.e., 
	\begin{equation}\label{optimal_condition}
	\nabla \bar{F}(X^{*})\T (X^{k+1}-X^{*}) \ge 0.
	\end{equation}
	Then, from~\eqref{dadmm_lamda}, \eqref{optimal_k2} \eqref{important_inequality}  and \eqref{optimal_condition}, we obtain
	\begin{align}
	\mu&\| X^{k+1}-X^{*}\|^2 \le  \nabla \bar{F}(X^{k+1})\T (X^{k+1}-X^{*})\nonumber\\& - \nabla \bar{F}(X^{*})\T (X^{k+1}-X^{*}) \nonumber\\
	=& (X^{k+1}-X^{*})\T [\lambda^{*}-\lambda^{k} -\rho (  X^{k+1} -  z^{k}) ] \nonumber\\
	=&  \rho(X^{*} - X^{k+1})\T(X^{k+1}-X^{k}) + (X^{k+1}-X^{*})\T\nonumber\\& \times[\rho(X^{k+1}-X^{k})+ \lambda^{*}-\lambda^{k} -\rho (  X^{k+1} -  z^{k}) ]\nonumber\\
	=& \rho(X^{*} - X^{k+1})\T(X^{k+1}-X^{k}) \nonumber\\&+ (X^{k+1}-X^{*})\T[\lambda^{*}-\lambda^{k} -\rho (  X^{k} -  z^{k}) ].\label{Pi}
	\end{align}
	From the equality law~\eqref{equality_law}, \eqref{Pi} changes to
	\begin{align}
	\mu&\| X^{k+1}-X^{*}\|^2 \le  \rho(X^{*} - X^{k+1})\T(X^{k+1}-X^{k}) \nonumber\\&- (X^{k+1}-X^{*})\T[\lambda^{k}-\lambda^{*} +\rho (  X^{k} -  z^{k}) ]\nonumber \\
	\le & \frac{\rho}{2}  ( \| X^{k} -X^{*} \|^2  - \| X^{k+1} -X^{*} \|^2) -\frac{\rho}{2}\| X^{k+1} -X^{k} \|^2 \nonumber\\& - (X^{k+1}-X^{*})\T[\lambda^{k}-\lambda^{*} +\rho (  X^{k} -  z^{k}) ].\label{Pi_1}
	\end{align}
	Denote 
\begin{equation}\label{defi_tao_k}
\tau_k\coloneqq  \frac{\rho}{2}\| X^{k} -X^{*} \|^2.
\end{equation}
	One can see that if $ \tau_k \rightarrow 0 $, then $ X^{k} \rightarrow X^{*} $, i.e., the optimal problem is solved.
	 Based on  \eqref{Pi_1}, we get
	\begin{align}
	(1+ \frac{2\mu}{\rho})\tau_{k+1}\le \tau_k -\frac{\rho}{2} \Omega_k,
	 \label{eq_linearConvergence_condition}
	\end{align}
	where $ \Omega_k \eqqcolon \| X^{k+1} -X^{k} \|^2 +  (X^{k+1}-X^{*})\T[\frac{2}{\rho}(\lambda^{k}-\lambda^{*}) +2 (  X^{k} -  z^{k}) ] $.
	In order to prove the linear convergence of our algorithm, we need to prove $ \Omega_k \ge 0 $. From~\eqref{defi_tao_k},
	we have 
%
	{\small
		\begin{align}
	\Omega_k 
	=& \| (X^{k+1}-X^{*}) -(X^{k}-X^{*}) \|^2 + [(X^{k+1}-X^{*})-(X^{k}\nonumber \\ 
	&-X^{*})]\T[\frac{2}{\rho}(\lambda^{k}-\lambda^{*}) +2 ( (X^{k}-X^{*}) -  (z^{k}-X^{*})) ]\nonumber \\ 
	& + {(X^{k}-X^{*})}\T[\frac{2}{\rho}(\lambda^{k}-\lambda^{*}) +2 (  (X^{k}-X^{*}) -  (z^{k}-X^{*})) ] \nonumber \\ 
	=& \|(X^{k+1}-X^{*}) -(X^{k}-X^{*}) + \frac{1}{\rho}(\lambda^{k}-\lambda^{*}) + ( (X^{k}-X^{*})\nonumber \\ 
	& -  (z^{k}-X^{*})) \|^2 + 2{(X^{k}-X^{*})}\T[\frac{1}{\rho}(\lambda^{k}-\lambda^{*})\nonumber \\ 
	& + (  (X^{k}-X^{*}) -  (z^{k}-X^{*})) ]  - \|  \frac{1}{\rho}(\lambda^{k}-\lambda^{*}) \nonumber \\ 
	&+ ( (X^{k}-X^{*}) -  (z^{k}-X^{*})) \|^2\nonumber \\ 
	=& \| (X^{k+1}-X^{*})  + \frac{1}{\rho}(\lambda^{k}-\lambda^{*}) -  (z^{k}-X^{*}) \|^2 \nonumber \\ 
	&+ \|X^{k}-X^{*}\|^2 - \|\frac{1}{\rho}(\lambda^{k}-\lambda^{*}) -  (z^{k}-X^{*})\|^2\label{Omega_k2} 
	\end{align}}
	By applying the following basic inequality~\cite{deng2016global}:
	\begin{equation}
	\| u+v  \|^2 \ge (1-\frac{1}{\mu_1}) \| u \|^2 + ( 1 - \mu_1 ) \| v  \|^2, \forall \mu_1 >0,
	\end{equation}
	\eqref{Omega_k2} changes to
		\begin{align}
	&\Omega_k 
	\ge (1-\frac{1}{\mu_1}) \|\frac{1}{\rho}(\lambda^{k}-\lambda^{*}) - (z^{k}-X^{*}) \|^2 + ( 1 - \mu_1 )\nonumber \\ 
	&\times\| X^{k+1} -X^{*}  \|^2 + \|X^{k}-X^{*}\|^2 \nonumber \\ 
	&- \|\frac{1}{\rho}(\lambda^{k}-\lambda^{*}) - (z^{k}-X^{*})\|^2\nonumber \\ =&\|X^{k}-X^{*}\|^2 + ( 1 - \mu_1 )\| X^{k+1} -X^{*}  \|^2\nonumber \\ 
	&-\frac{1}{\mu_1} \|\frac{1}{\rho}(\lambda^{k}-\lambda^{*}) - (z^{k}-X^{*}) \|^2 \\ =&\frac{2}{\rho}\tau_{k}+ \frac{2( 1 - \mu_1 )}{\rho} \tau_{k+1}-\frac{1}{\mu_1} \|\frac{1}{\rho}(\lambda^{k}-\lambda^{*}) - (z^{k}-X^{*}) \|^2. \nonumber
	\end{align}
Consequently, \eqref{eq_linearConvergence_condition}  changes to
{\small
\begin{align}
\| X^{k+1} &-X^{*} \| \le \sqrt{\frac{\rho}{\mu_1[(2-\mu)\rho+ 2\mu]}}\|\frac{1}{\rho}(\lambda^{k}-\lambda^{*}) \nonumber \\ &- (z^{k}-z^{*}) \| \label{Xk_final}\\ 
\le& \sqrt{\frac{\rho}{\mu_1[(2-\mu)\rho+ 2\mu]}} (\|\frac{1}{\rho}(\lambda^{k}-\lambda^{*}) \|+\| (z^{k}-z^{*}) \|).\nonumber
\end{align}}
Based on Section 3.4 in~\cite{deng2016global}, $ (z^{k}, \lambda^{k}) $ converges Q-linearly. Then, based on the definition of R-linear convergence, from \eqref{Xk_final}, $ x^{k} $ converges approximately R-linearly.
\end{proof}
\begin{remark}
	Only the Q-linear convergence of $ (z^{k}, \lambda^{k}) $ is proved in~\cite{deng2016global}. In this paper, by \eqref{Xk_final}, though $ x^{k} $ converges not exactly R-linearly, we prove it converges  approximately R-linearly.
\end{remark}
%
%
%
%
\section{Examples}\label{sec:examples}

In this section, two examples are presented to demonstrate the effectiveness and performance of Algorithms~\ref{algorithm:1} (in Sec.~\ref{sec_alg1}) and~\ref{algorithm:2} (in Sec.~\ref{sec_alg2}).

\subsection{Distributed least square problem}\label{sec_alg1}
The distributed least square problem is considered as
\vspace{-0.2cm}
\begin{equation}\label{problem_lesatSquare}
\operatorname*{argmin}_{x\in \mathbb{R}^{p}}  f(x) = \frac{1}{2}\sum_{i=1}^{n} \|A_ix-b_i\|^2,
\end{equation}
where $ A_i  \in \mathbb{R}^{q\times p} $ is only known to node $ v_i $, $ b_i \in \mathbb{R}^{q} $ is the measured data and $ x \in \mathbb{R}^{ p} $ is the common decision variable that needs to be optimized. For the automatic generation of large number of different matrices $ A_i $, we choose $ q = p $ to have the square $ A_i $. All elements of $ A_i $ and $ b_i $ are set from independent and identically distributed (i.i.d.) samples of standard normal distribution $ \mathcal{N}(0,1) $.
To better demonstrate the difference between our proposed Algorithm~\ref{algorithm:1} and the algorithm in \cite{khatana2020d}, we set $ p=3 $.

\begin{figure}[t]
	\includegraphics[width=\columnwidth]{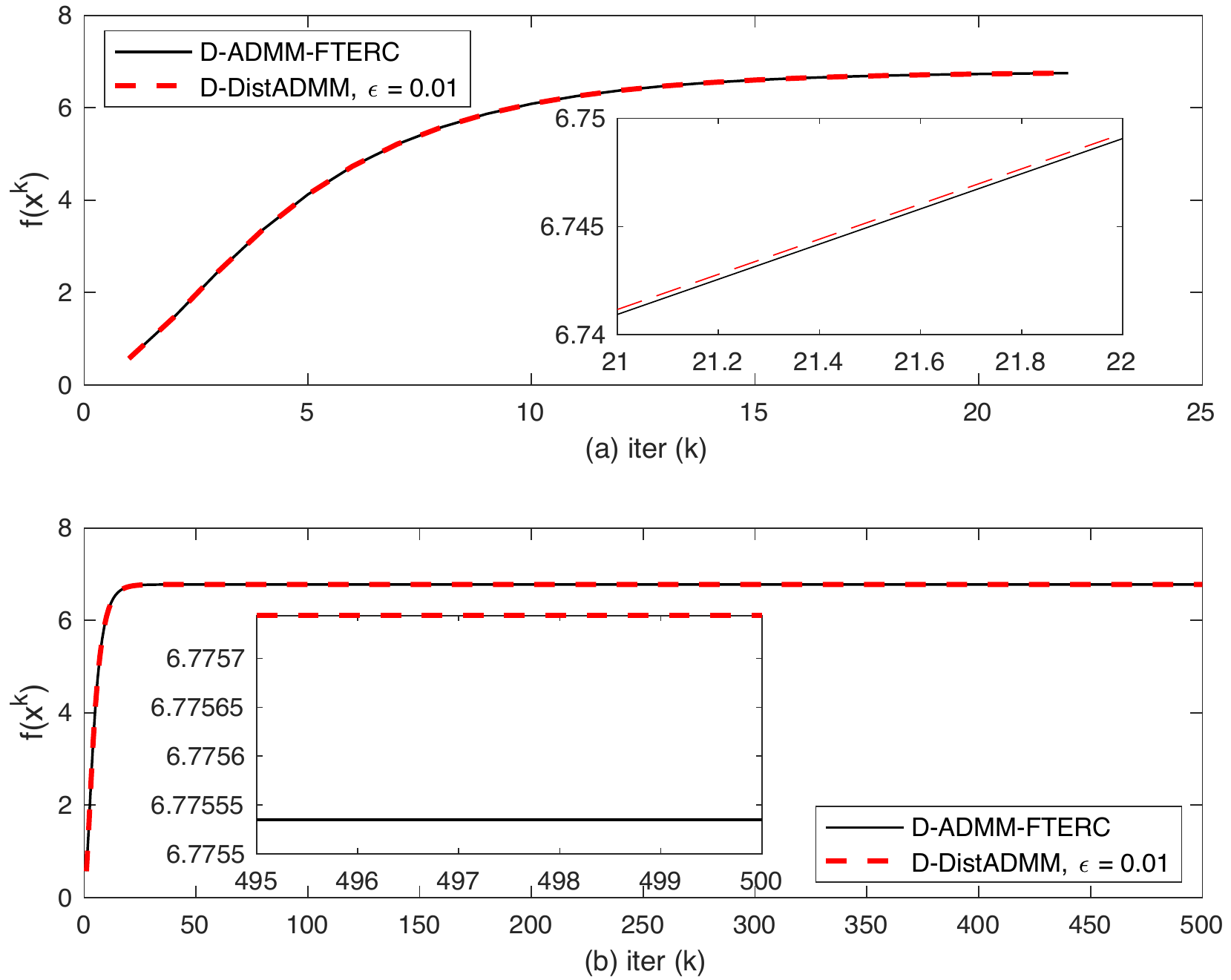}
	\vspace{-0.7cm}
	\caption{Comparison of solutions to problem~\eqref{problem_lesatSquare} between D-ADMM-FTERC in this paper and D-DistADMM in~\cite{khatana2020d}: (a) with stopping condition (steps 14-16 in Algorithm~\ref{algorithm:1} with  absolute tolerance  $ = 1\exp{-4}$ and
		relative tolerance   $= 1\exp{-2} $); (b) without stopping condition (herein $ k_{\max} = 500 $).}
	\label{fig_obj_comparison}
\end{figure}
\begin{figure}[t]
	\includegraphics[width=\columnwidth]{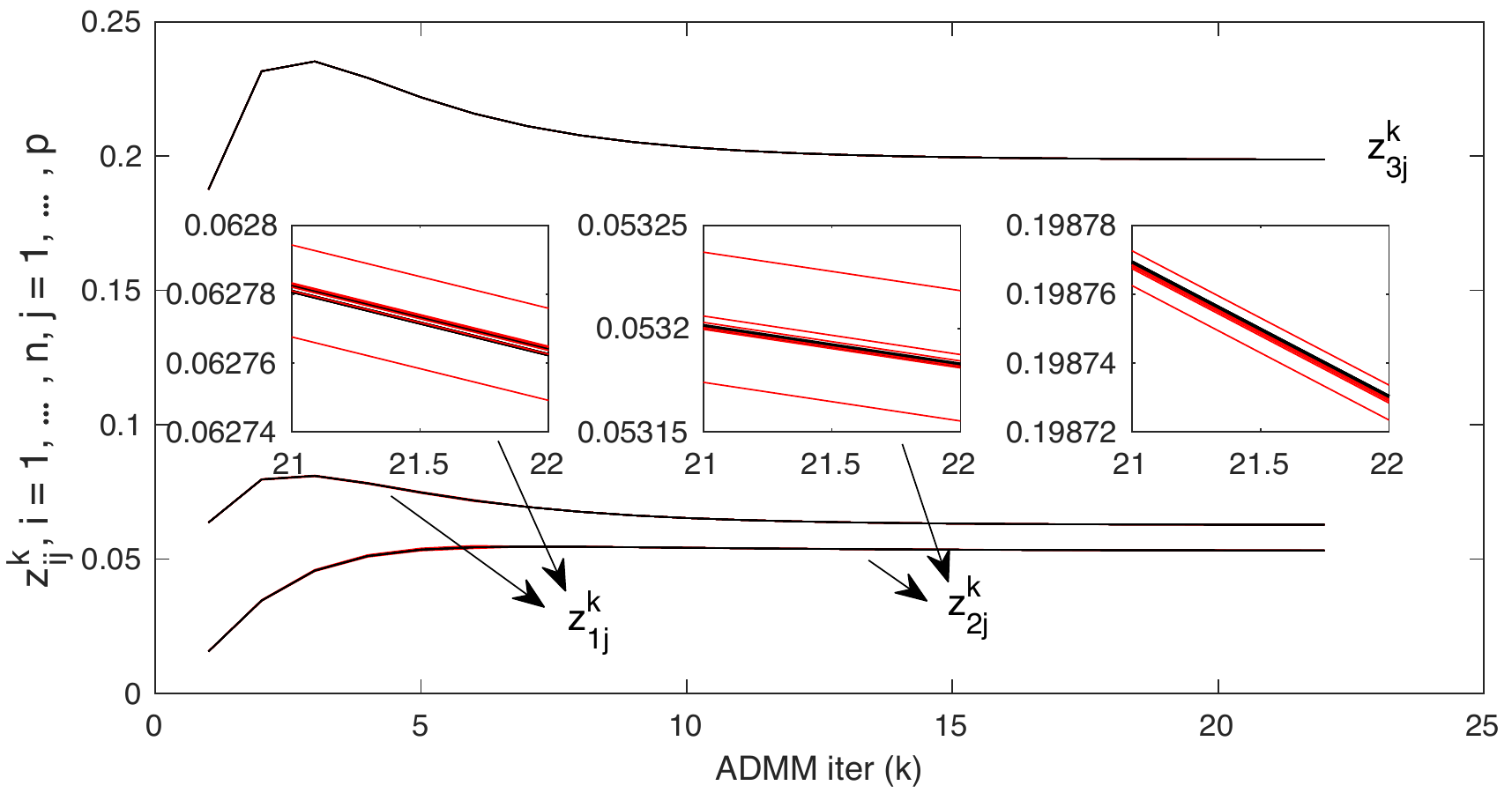}
	\vspace{-0.7cm}
	\caption{Technique comparison of FTERC and finite-time $ \epsilon $ consensus in~\cite{khatana2020d}.  $ z^{k+1} $ update in~\eqref{dadmm_z}: one can see clearly that $ z_{ij}, i = 1,\ldots,n $ in FTERC (black color) converge to the same value during each ADMM iteration while stay inside a bound ($ \|z_i-z_j\|\le \epsilon $) in finite-time $ \epsilon $ consensus (red color). 
	}
	\label{fig_z_rc_step}
\end{figure}
First, we choose $ n=6 $ to have 6 nodes having a strongly connected digraph. Fig.~\ref{fig_obj_comparison} shows that the solution of D-ADMM-FTERC is always smaller than the one of D-DistADMM in~\cite{khatana2020d} no matter whether we have ADMM stopping condition or not, which verifies Remark~\ref{remark_comparison}. Note the value of $ \epsilon = 0.01 $ is already very small and $ \epsilon $ cannot be zero. If it is the case, the finite-time $ \epsilon $ consensus technique in D-DistADMM will become normal ratio consensus technique in Section~\ref{preliminary_rc}, making the $ z^{k+1} $ update in~\eqref{dadmm_z} converges in infinite time as it asks $ z_i = z_j $ (from $ \|z_i-z_j\|\le \epsilon $). 
Fig.~\ref{fig_z_rc_step} gives more details about the comparisons. 


Now, we choose $ n=70, 100 $ and $ 700 $, respectively, with a random strongly connected digraph. Here, we choose $ n' = n+1 $ from  the FERTC structure in Fig.~\ref{FTERC_structure} which is verified by Fig~\ref{fig_z_rc_stepN100}.
In addition,
Fig~\ref{fig_z_rc_stepN100} demonstrates that after the first 2 ADMM iterations,  FTERC iteration number in each ADMM iteration is much smaller than finite-time $ \epsilon $ consensus in~\cite{khatana2020d}.
Note that during the first two ADMM optimization steps, finite-time $ \epsilon $ consensus has less iteration numbers. This is related to the node number.
Furthermore, Table~\ref{tab_timeComparison} describes the running time comparison on an Intel Core i5 processor at 2.6 GHz with Matlab R2020b for $ \epsilon = 0.01, k_{\max}=200 $ for this example. One can see D-ADMM-FTERC is time-efficient than D-DistADMM, and is much more time-efficient for large scale systems.

\begin{figure}[h]
	\includegraphics[width=0.95\columnwidth]{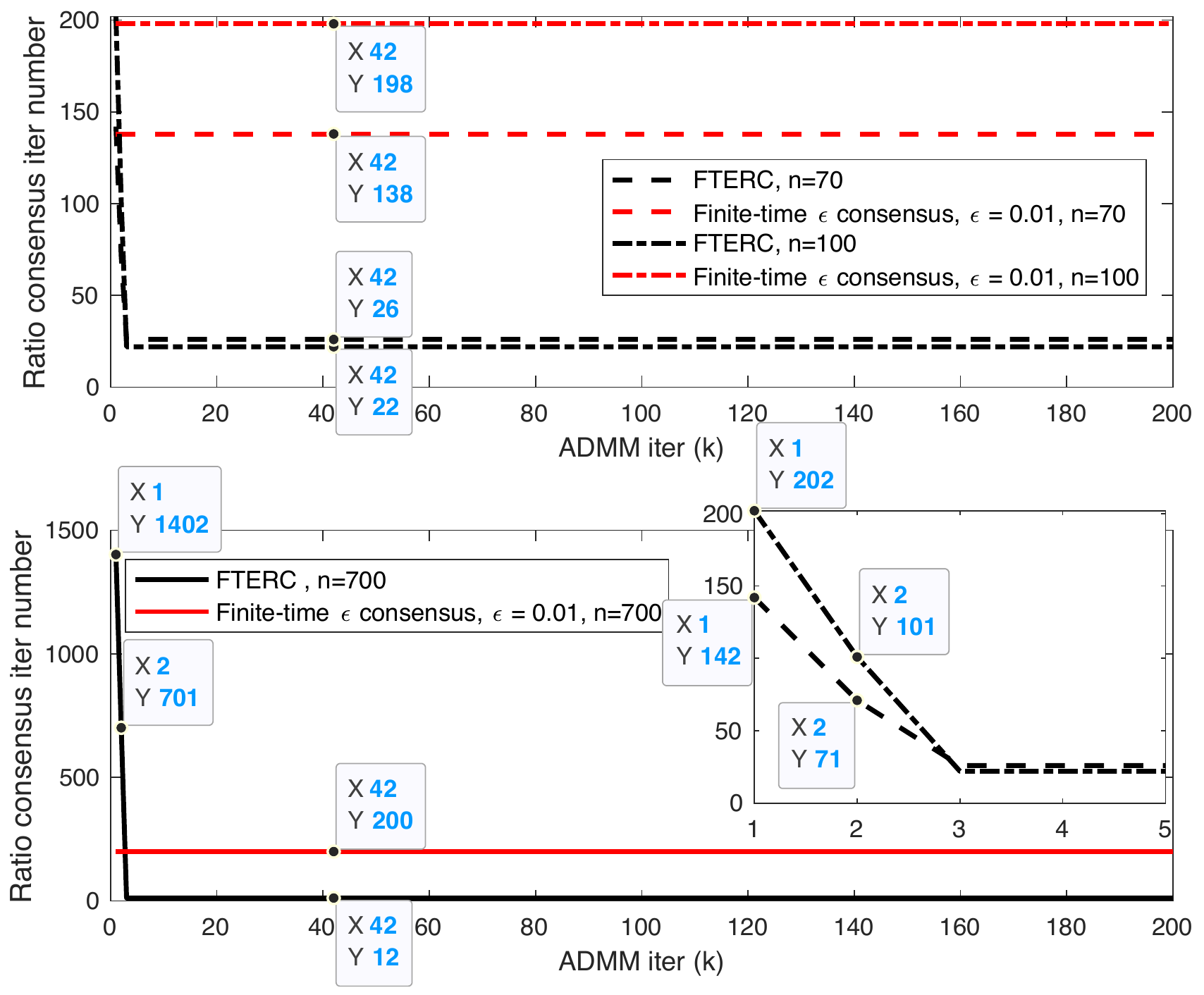}
	\caption{ Ratio consensus iteration number comparison: $ n=70, 100, 700 $.
	}
	\label{fig_z_rc_stepN100}
\end{figure}

\begin{table}[h]
	\centering
	\caption{Running time comparison.}
	\label{tab_timeComparison}
	{\small
		\begin{tabular}{clll}
			\hline
			\multicolumn{1}{|c|}{}             & \multicolumn{1}{c|}{n=70}       & \multicolumn{1}{c|}{n=100}      & \multicolumn{1}{l|}{n=700}        \\ \hline
			\multicolumn{1}{|c|}{D-ADMM-FTERC} & \multicolumn{1}{l|}{1.4122s}  & \multicolumn{1}{l|}{2.1967s}  & \multicolumn{1}{l|}{52.0160s}    \\ \hline
			\multicolumn{1}{|c|}{D-DistADMM}   & \multicolumn{1}{l|}{22.5361s} & \multicolumn{1}{l|}{78.7964s} & \multicolumn{1}{l|}{30867.5860s} \\ \hline
			&                                 &                                 &                                 
	\end{tabular}}
\end{table}

\subsection{Distributed $ l_1 $ regularized logistic regression problem}\label{sec_alg2}
The background of {distributed $ l_1 $ regularized logistic regression} is elaborated in Sec. 11.2 of \cite{boyd2011distributed}. The problem is
\begin{equation}\label{problem_l1}
\operatorname*{argmin}_{x\in \mathbb{R}^{p}, v\in \mathbb{R}}  \sum_{i=1}^{m} \log(1+ \exp (-b_i (a_i^Tx + v) )) + \mu \|x\|_1,
\end{equation}
where the training set consists of $ m $ pairs $ (a_i, b_i) $ with $ a_i\in \mathbb{R}^{p} $ being a feature vector and $ b_i\in \{-1,1\}$ being the corresponding label. For problem~\eqref{problem_l1}, we generate $ m = 2\times 10^4 $ training examples and $ p=200 $ features. The $ m $ examples are distributed among $ n=100 $ subsystems (processors or nodes) among which a strongly connected communication graph (topology) can be built.

The `true' vector $ x^{true} \in \mathbb{R}^{p} $ has $ 100 $ normally distributed nonzero entries. The true intercept $ v^{true} \in \mathbb{R} $ is also sampled independently from a standard normal distribution. Labels $ b_i $ are generated by $ b_i = \textbf{sign}(a_i^T x^{true}+ v^{true} +v_i), v_i \in \mathcal{N}(0,0.1) $.
The regularization parameters $ \mu $ is set as $ \mu = 0.1 \mu_{\max}  $, where $\mu_{\max}  $ is the critical value above which the solution of the problem is $ x^{*} = 0 $. The Sec. 11.2 of \cite{boyd2011distributed} presents details how to calculate $ \mu_{\max} $.

Denote $ f_i(x_i) \coloneqq \log(1+ \exp (-b_i (a_i^Tx_i + v) )), r(x) \coloneqq  \mu \|x\|_1 $. To solve problem~\eqref{problem_l1}, in Sec. 8.2 of \cite{boyd2011distributed}, fitting the model involves solving the global consensus problem in the following with local variables $ x_i = (v_i, w_i) $ and global consensus variable $ z = (v,w) $:
\begin{equation}\label{regulation_admm_problem}
\begin{aligned}
\min \, &\sum_{i=1}^{n} f_i(x_i)+r(z), \\
\text{s.t.}  \, &x_i = z, i = 1, \ldots, n.
\end{aligned}
\end{equation}
If $ r $ is fully separable,  the generic global variable consensus ADMM algorithm described in Sec. 7.1 of \cite{boyd2011distributed} can be applied with scaled dual variable as 
\begin{align}
x_i^{k+1} =&  \operatorname*{argmin}_{x_i} \left(  f_i(x_i) +  (\rho/2)\|x_i - z^{k} + u_i^k \|_2^{2} \right), \label{dadmm_x1}\\
z^{k+1} =&  \operatorname*{argmin}_z \left(  r(z) +(n\rho/2)\|z- \bar x^{k+1} - \bar \lambda^{k} \|_2^{2}\right), \label{dadmm_z1}\\
\lambda_i^{k+1} =& \lambda_i^{k} +  x_i^{k+1} - z^{k+1} \label{dadmm_lamda1},
\end{align}
where $ \bar x $ is the average of $ x_1, \ldots, x_n $ and $ \bar \lambda $ is the average of $ \lambda_1, \ldots, \lambda_n $. 

The $ z $-update includes an averaging step (i.e., $ \bar x, \bar \lambda $ ) and a followed proximal step involving $ r $ which is a threshold operation:
\begin{equation}
z^{k+1} \coloneqq S_{\mu/(n\rho)} (\bar x^{k+1} + \bar \lambda^{k}).
\end{equation}
It is worth noting that the $ z $-update needs a central collector to calculate $ \bar x^{k+1} $ and  $ \bar \lambda^{k} $ in each ADMM step $ k $ and it assumes each subsystem has the same global variable $ z $. A central collector means an extra subsystem (node) is needed to collect data from all other subsystems, do the average (or even the whole $ z $-update because $ z $ is global) calculation and broadcast the calculated results to those subsystems, which is one centralized step, not distributed.


To make the $ z $-update distributed, our proposed FTERC or {FTDT} can be applied here. Unlike \eqref{regulation_admm_problem}, we formulate the local consensus problem as follows:
\begin{equation}\label{regulation_admm_problem2}
\begin{aligned}
\min \, &\sum_{i=1}^{n} f_i(x_i)+r(z_i), \\
\text{s.t.}  \, &x_i = z_i, i = 1, \ldots, n,
\end{aligned}
\end{equation}
which is similar as \eqref{problem_reformulated2}. Denote the vector $ Z \coloneqq [z_1^T, z_2^T, \ldots, z_n^T]^T, Z\in \mathbb{R}^{np}  $.  The $ z $-update changes to
\begin{equation}\label{zUpdate_fterc}
Z^{k+1} =  \operatorname*{argmin}_Z \left(  r(Z) + g(Z) +(n\rho/2)\|Z- \bar x^{k+1} - \bar \lambda^{k} \|_2^{2}\right), \\
\end{equation}
where $ g(Z) $ is defined in~\eqref{indicator_function}.
The solution of \eqref{zUpdate_fterc} also takes two steps.
First, by forcing $ Z $ to go into set $ \mathcal{C} $ defined in~\eqref{setC}, we use FTERC or {FTDT} to distributively get the average $ \bar z_i^k \coloneqq \bar x^{k+1} + \bar \lambda^{k}$ for each subsystem $ i $. 
Then, \eqref{zUpdate_fterc} changes to $ z_i^{k+1} =  \operatorname*{argmin}_{z_i} \left(  r(z_i) +(n\rho/2)\|z_i- \bar z_i^k \|_2^{2}\right), $ and the threshold operator is used to get the solution $ z_i^{k+1} \coloneqq S_{\mu/(n\rho)} (\bar z_i^k) $.

Same as in \cite{boyd2011distributed}, we use L-BFGS-B to carry out $ x $-updates~\eqref{dadmm_x1} distributively (in parallel). The $ \lambda $-updates are also the same as \eqref{dadmm_lamda1}. The Matlab codes for the algorithm \eqref{dadmm_x1},
\eqref{dadmm_z1} and \eqref{dadmm_lamda1} in~\cite{boyd2011distributed} are the example 10 in \url{https://web.stanford.edu/~boyd/papers/admm/}.

Based on our proposed FTERC or {FTDT}, both $ x $-update and $ z $-update are distributed, while in the previous codes only $ x $-update is distributed.
This is useful either when there are so many training examples that it is inconvenient or impossible to process them on a single machine or when the data is naturally collected or stored in a distributed fashion. This includes, for example, on-line social network data, web server access logs, wireless sensor networks, and many cloud computing applications more generally.

Fig.~\ref{fig_l1_regulation} (a) shows that the above two algorithms and D-DistADMM in~\cite{khatana2020d} can reach the minimization objective as we can see there is nearly no obvious difference in the optimal value. 
However, it is worth noting that the above fully distributed advantage from {F}D-ADMM-FTDT and  D-DistADMM also comes with one cost: running time. By running this example on an Intel Core i5 processor at 2.6 GHz with Matlab R2020b for $ \epsilon = 0.05 $, the running time for above three algorithms are respectively 3.691185s, 238.563915s, 1173.344407s. One can see even though FD-ADMM-FTDT running time is much larger than the algorithm in \cite{boyd2011distributed} in which the average calculation is done with a simple \textit{mean} function in Matlab by the central collector, it is nearly 5 times less then D-DistADMM which is because the running ratio consensus steps is much less shown in Fig.~\ref{fig_l1_regulation} (b). 

\begin{figure}[h]
	\centering
	\includegraphics[width=\columnwidth]{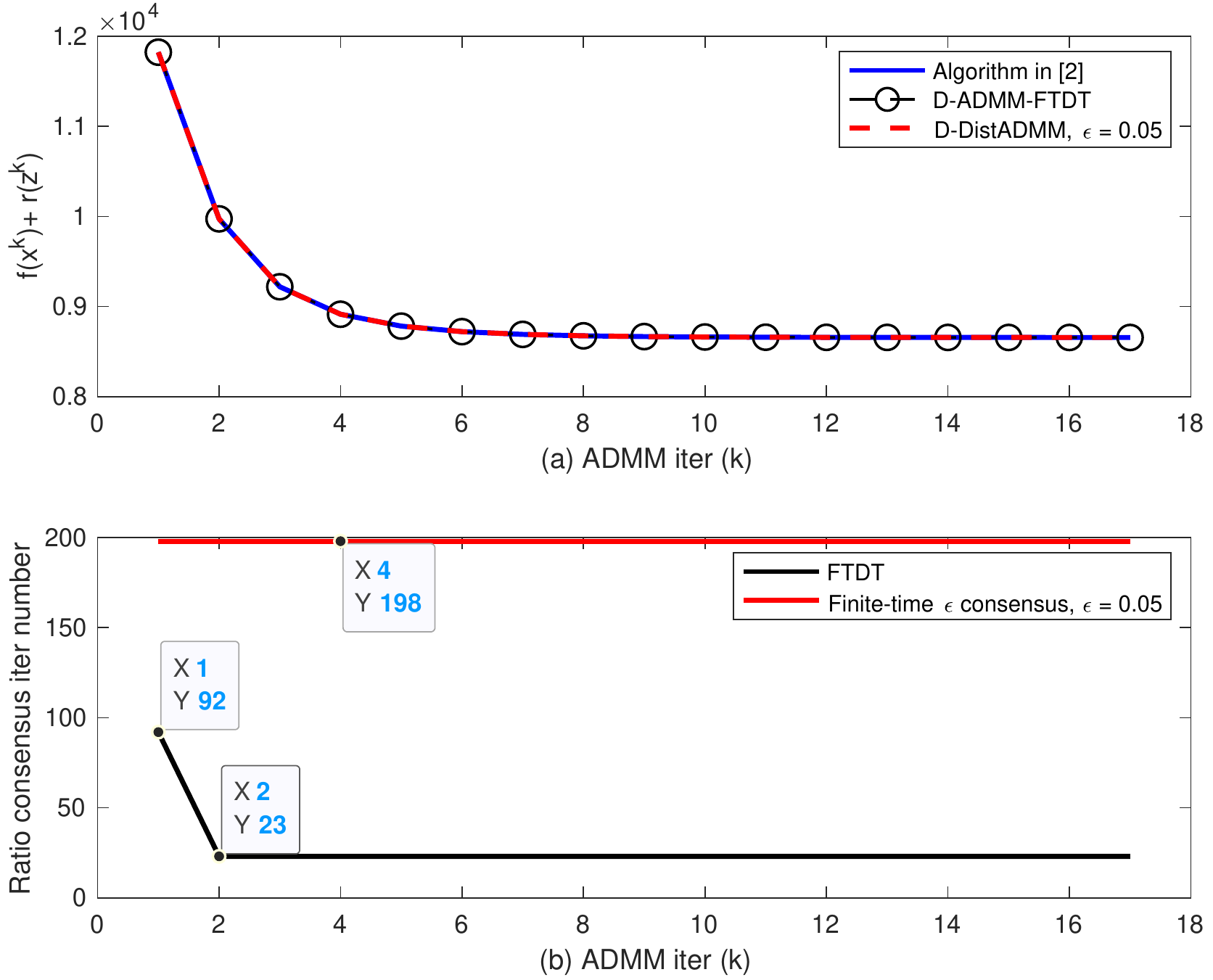}
	\caption{Comparison of solutions to problem~\eqref{problem_l1} among the algorithm in \cite{boyd2011distributed}, FD-ADMM-FTDT in this paper and D-DistADMM in~\cite{khatana2020d} with absolute tolerance  $  1\exp^{-4}$ and
		relative tolerance   $ 1\exp^{-2} $: (a) minimized objective value; (b) ratio consensus running number.}
	\label{fig_l1_regulation}
\end{figure}


%
%
%
%
\section{Conclusions and Future Directions}\label{sec:conclusions}

\subsection{Conclusions}
A distributed alternating direction method of multipliers using finite-time exact ratio consensus (D-ADMM-FTERC) algorithm is proposed  to solve the multi-node convex optimization problem under digraphs.
Compared to other state of art distributed ADMM algorithms, D-ADMM-FTERC can not only apply to digraphs and reach the exact optimal solution, but also is very time-efficient, especially for large-scale systems.
\subsection{Future Directions}
This work assumes that the nodes are aware of an upper bound of the size of the network, which in most applications consisting of static networks (e.g., in data centers) this is readily available. However, in more dynamical networks (e.g., sensor networks) this can become a limitation. We plan to find ways to terminate the consensus process in a distributed fashion without any knowledge of the size or the structure of the network.

Furthermore, it would be interesting to study the asynchronous implementation of our proposed method.

%
%
%
%

\bibliographystyle{IEEEtran}
\bibliography{bib_finite}

\end{document}